\documentclass[12pt,a4paper]{amsart}
\usepackage{amssymb,mathrsfs,stmaryrd}

\newtheorem{theorem}{Theorem}
\newtheorem{lemma}{Lemma}[section]
\newtheorem{corollary}{Corollary}[section]
\newtheorem{question}{Problem}

\theoremstyle{definition}
\newtheorem{remark}{Remark}
\newtheorem*{remark-nonum}{Remark}
\newtheorem*{example}{Example}

\numberwithin{equation}{section}

\newcommand{\vertbar}{\>|\>}
\newcommand{\set}[2]{\ensuremath{\{ #1 \vertbar #2 \}}}

\def\liebrack{\ensuremath{[\,\cdot\, , \cdot\,]}}

\DeclareMathOperator{\ad}{ad}
\DeclareMathOperator{\C}{C}
\DeclareMathOperator{\Cent}{Cent}
\DeclareMathOperator{\diag}{diag}
\DeclareMathOperator{\End}{End}
\DeclareMathOperator{\HC}{HC}
\DeclareMathOperator{\Hom}{Hom}
\DeclareMathOperator{\HomLie}{HomLie}
\DeclareMathOperator{\Homol}{H}
\DeclareMathOperator{\Ker}{Ker}
\DeclareMathOperator{\R}{R}
\DeclareMathOperator{\Tr}{Tr}
\DeclareMathOperator{\Z}{Z}

\begin{document}

\title{A compendium of Lie structures on tensor products}
\author{Pasha Zusmanovich}
\address{}
\email{justpasha@gmail.com}
\date{First written August 22, 2012; last minor revision February 5, 2014}
\thanks{\textsf{arXiv:1303.3231}; 
Zapiski Nauchn. Sem. POMI \textbf{414} (2013) (N.A. Vavilov Festschrift), 
40--81}

\begin{abstract}
We demonstrate how a simple linear-algebraic technique used earlier to compute 
low-degree cohomology of current Lie algebras, can be utilized to compute
other kinds of structures on such Lie algebras, and discuss further 
generalizations, applications, and related questions.
While doing so, we touch upon such seemingly diverse topics as 
associative algebras of infinite representation type, Hom-Lie structures,  
Poisson brackets of hydrodynamic type, Novikov algebras, 
simple Lie algebras in small characteristics, and Koszul dual operads.
\end{abstract}

\maketitle

\setcounter{footnote}{1}

\section*{Introduction}

Earlier, we have demonstrated how a simple 
linear-algebraic technique, somewhat resembling separation of variables
in differential equations, allows to obtain general results about the 
low-degree cohomology of current Lie algebras, i.e. Lie algebras of the form 
$L \otimes A$, where $L$ is a Lie algebra, and $A$ is an associative 
commutative algebra, with an obvious Lie bracket
$$
[x \otimes a, y \otimes b] = [x,y] \otimes ab
$$
for $x,y \in L$, $a,b\in A$. On a pedestrian level, kept through the most of 
\cite{low} and \cite{without-unit}, this technique amounts to symmetrization of
cocycle equations with respect to variables belonging to $L$ and $A$ separately.
On a more sophisticated level, the method can be formulated in terms of a 
certain spectral sequence defined with the help of Young symmetrizers on the 
underlying spaces $L$ and $A$ (described briefly in \cite[\S 4]{low}).

In this, somewhat eclectic, paper we discuss further generalizations, 
applications and extensions of this method. We start in \S \ref{sec-fin} by 
providing examples of finite-dimensional Lie algebras having an infinite number
of cohomologically nontrivial non-isomorphic modules of a fixed finite 
dimension. This answers a question and corrects a statement from an (old) paper
by Dzhumadil'daev \cite{dzhu-al}.
The Lie algebras in our example are current Lie algebras, and our technique 
allows to reduce the question about the number of non-isomorphic cohomologically
nontrivial ``current'' modules over such algebras to a similar, but purely
representation-theoretic (i.e. not appealing to any cohomological condition) 
question about the number of non-isomorphic modules over the corresponding 
associative commutative algebra $A$.

Further, we consider a question of description of Poisson structures on a given 
Lie algebra -- i.e., description of all (not necessary associative) algebra 
structures which, together with a given Lie structure, form a Poisson algebra. 
The Poissonity condition resembles the cocycle equation, and this allows to 
apply the same methods used previously to compute low-degree cohomology of 
current Lie algebras. In \S \ref{sec-poisson-current} we describe Poisson 
structures on current and Kac--Moody Lie algebras, and in 
\S \ref{sec-poisson-gl} -- on Lie algebras of the form $\mathsf{sl}_n(A)$. 
Associative Poisson structures on Kac--Moody algebras were described earlier by
Kubo in \cite{kubo-kac-moody} via case-by-case computations, and we generalize 
and provide a conceptual explanation of Kubo's result.

In \S \ref{sec-homlie} we compute, in a similar way, Hom-Lie structures 
on current Lie algebras.

In the last \S \ref{sec-dual} we discuss a possibility of extension of all the 
previous results from current Lie algebras to Lie algebras obtained from the 
tensor product of algebras over Koszul dual binary quadratic operads, with a 
special emphasis on Novikov algebras. This is the most speculative part of the 
paper: it does not contain a single theorem, but numerous examples, 
speculations, and questions, including an observation that the classical 
construction of ``Poisson brackets of hydrodynamic type'' due to 
Gelfand--Dorfman and Balinskii--Novikov, arises in this context. (However, this
simple observation is, perhaps, the most important result of the paper).

\bigskip

This is a modest contribution in honor of Nikolai Vavilov, an amazing man and 
amazing mathematician. There is an obvious similarity between current Lie 
algebras, the main heroes of this paper, and linear and close to them groups 
over rings, Vavilov's favorite objects of study. The latters, however, offer 
much deeper difficulties, and it is my hope to extend in the future the methods
and ideas of this paper to be able to tackle some questions close to Vavilov's 
heart.

\section*{Notation and conventions}

Our notation is mostly standard.
The base field is denoted by the letter $K$.
By $\Homol^n(L,M)$ is denoted the $n$th degree Chevalley-Eilenberg cohomology of
a Lie algebra $L$ with coefficients in an $L$-module $M$. The symbol
$\Hom_A$ for a (Lie or associative) algebra $A$ is understood in the category 
of $A$-modules (in particular, $\Hom_K$ is just the vector space 
of linear maps between its arguments). 
The notation $\End_A(V)$ serves as a shortcut for $\Hom_A(V,V)$.
All unadorned tensor products are assumed over $K$. Direct sum $\oplus$ is
always understood in the category of vector spaces over $K$.
An operator of multiplication by an element $u$ in an associative commutative
algebra is denoted by $\R_u$. For a (not necessary associative) algebra $A$,
$A^{(-)}$ denotes its \emph{skew-symmetrization}, i.e. an algebra structure
on the same underlying vector space $A$ subject to the new multiplication
$[a,b] = ab - ba$ for $a,b \in A$. An algebra is called \emph{Lie-admissible},
if its skew-symmetrization is a Lie algebra.

All other necessary notation and notions are defined as they are introduced in 
the text.

\section{On the finiteness of the number of cohomologically nontrivial Lie 
modules}\label{sec-fin}

In 1980s, Askar Dzhumadil'daev has initiated a study of cohomology of
finite-dimensional Lie algebras over fields of positive characteristic,
which showed a (long time anticipated) drastic difference with the 
characteristic zero case. It turned out that cohomology in positive 
characteristic does not vanish much more often that in the characteristic zero
case (for example, Whitehead lemmas do not hold), and a natural question arose 
about the number of cohomologically nontrivial modules. 

One of the papers of that period is \cite{dzhu-al}, and we start our discussion
by indicating one inaccuracy there. In \cite[Theorem 1]{dzhu-al} it is claimed 
that, for any finite-dimensional Lie algebra $L$ over a field of positive 
characteristic, the number of cohomologically nontrivial non-isomorphic 
$L$-modules of a fixed finite dimension is finite. 
There is a following gap in the proof of this statement: when arguing 
inductively about indecomposable modules over a Lie algebra $L$, it is claimed 
that non-equivalent extensions of $L$-modules $0 \to V_1 \to V \to V_2 \to 0$ 
are described by the space $\Homol^1(L, \Hom_K(V_1,V_2))$ which is 
finite-dimensional, and hence the set of such extensions is finite provided the
set of modules $V_1$ and $V_2$ is finite.
This is obviously incorrect: the extensions are in bijective correspondence with
the space $\Homol^1(L, \Hom_K(V_1,V_2))$ as a \textit{set}, which is infinite
(provided the base field $K$ is infinite).

This does not, however, exclude the possibility that the set of 
\textit{non-isomorphic} modules would be finite, as there can be ad hoc 
module isomorphisms which do not follow from the equivalence of modules. 

In the category of associative algebras, where extensions of modules are 
described, up to equivalence, by the corresponding first Hochschild cohomology,
it is known that algebras can have both finite and infinite number of 
indecomposable modules of a fixed finite dimension.

The situation in the category of finite-dimensional Lie $p$-algebras was studied
in \cite{fs}. It turns out that Lie $p$-algebras having a finite number of 
indecomposable $p$-modules of a fixed finite dimension, form quite a narrow
class.

Still, from all this it is not clear whether the statement of 
\cite[Theorem 1]{dzhu-al} is true. The aim of this section is to show that it is
not: we provide a recipe to construct Lie algebras, of arbitrarily
high finite dimension, having an infinite number of non-isomorphic 
cohomologically nontrivial modules of a fixed finite dimension. These Lie 
algebras come as current Lie algebras $L\otimes A$, and the modules appear as 
the corresponding ``current modules'', i.e. $L\otimes A$-modules of the form 
$M\otimes V$, where $M$ is an $L$-module and $V$ is an $A$-module, with an 
obvious action
$$
(x\otimes a) \bullet (m\otimes v) = (x\bullet m) \otimes (a\bullet v)
$$
for $x\in L$, $a\in A$, $m\in M$, $v\in V$; the bullets denote, by abuse of 
notation, the respective module actions.
By choosing suitably $L$ and $M$, we are able to reduce the question to those
about the finiteness of the number of modules of a given dimension over $A$, 
i.e. to a similar (but simpler, as it does not involve any cohomological 
condition) question in the category of associative commutative algebras.

This construction works in any characteristic of the base field,
so it also answers affirmatively Question 1 from \cite{dzhu-al}: does there exist
finite-dimensional Lie algebras over a field of characteristic zero having
an infinite number of cohomologically nontrivial non-isomorphic modules of a 
given finite dimension?

\begin{lemma}\label{embed}
Let $L$ be a Lie algebra, $A$ an associative commutative algebra with unit, 
$M$ an $L$-module, and $V$ an unital $A$-module. Then there is an embedding of 
$\Homol^*(L,M) \otimes V$ into $\Homol^*(L\otimes A, M\otimes V)$.
\end{lemma}

\begin{proof}
Define a map of the Chevalley--Eilenberg cochain complexes
$$
\xi: \C^n(L,M) \otimes V \to \C^n(L\otimes A, M\otimes V)
$$
as follows:
$$
\xi(\varphi \otimes v) (x_1\otimes a_1, \dots, x_n\otimes a_n) =
\varphi(x_1, \dots, x_n) \otimes (a_1 \cdots a_n) \bullet v
$$
where $\varphi\in \C^n(L,M)$, $v\in V$, $x_i\in L$, $a_i\in A$.

Let us see how differentials in the complexes 
$\C^*(L\otimes A, M\otimes V)$ and $\C^*(L,M)$, denoted as $d_{L\otimes A}$ and 
$d_L$ respectively, interact with this map:
\begin{align*}
d_{L\otimes A} & \xi(\varphi \otimes v) (x_1 \otimes a_1, \dots, x_{n+1} \otimes a_{n+1}) 
\\ &= 
\>\>\>\>\> \sum_{i=1}^{n+1} \>\>\>\>\>\>\>\> (-1)^{i+1} (x_i \otimes a_i) \bullet 
\xi(\varphi \otimes v) (x_1 \otimes a_1, \dots, \widehat{x_i \otimes a_i}, \dots, 
x_{n+1} \otimes a_{n+1}) \\ &+
\sum_{1 \le i < j \le n+1} (-1)^{i+j}
\xi(\varphi \otimes v) ([x_i \otimes a_i, x_j \otimes a_j], x_1 \otimes a_1, \dots, \widehat{x_i \otimes a_i},
\dots, \widehat{x_j \otimes a_j}, \\ & \hskip 212pt \dots, 
x_{n+1} \otimes a_{n+1})
\\ &=
\>\>\>\>\> \sum_{i=1}^{n+1} \>\>\>\>\>\> (-1)^{i+1} 
x_i \bullet \varphi(x_1, \dots, \widehat{x_i}, \dots, x_{n+1}) \otimes 
(a_i a_1 \cdots \widehat{a_i} \cdots a_{n+1}) \bullet v \\ 
&+
\sum_{1 \le i < j \le n+1} (-1)^{i+j}
\varphi ([x_i,x_j], x_1, \dots, \widehat{x_i}, \dots, \widehat{x_j}, \dots, x_{n+1}) 
\\ & \qquad\qquad\qquad\qquad\quad \otimes (a_i a_j a_1 \cdots \widehat{a_i} \cdots \widehat{a_j} \cdots a_{n+1}) \bullet v 
\\ & =
d_L \varphi(x_1, \dots, x_{n+1}) \otimes (a_1 \cdots a_{n+1}) \bullet v 
=
\xi(d_L \varphi \otimes v) (x_1 \otimes a_1, \dots, x_{n+1} \otimes a_{n+1}) .
\end{align*}

Thus $\xi$ is a homomorphism of cochain complexes, and hence it induces the 
map between their cohomology:
$$
\overline\xi: \Homol^*(L,M) \otimes V \to \Homol^*(L\otimes A, M\otimes V) .
$$
The kernel of $\overline\xi$ consists of classes of cocycles
$\sum_k \varphi_k \otimes v_k \in \Z^n(L,M) \otimes V$ such that 
$$
\sum_k \varphi_k (x_1, \dots, x_n) \otimes (a_1 \cdots a_n) \bullet v_k = 
d_{L\otimes A} \Omega (x_1\otimes a_1, \dots, x_n\otimes a_n)
$$ 
for some $\Omega \in \C^{n-1}(L\otimes A, M\otimes V)$ and all 
$x_1, \dots, x_n \in L$, $a_1, \dots, a_n \in A$. Setting in this equality all
$a_i$'s to $1$, we get
\begin{equation}\label{yoyo13}
\sum_k \varphi_k (x_1, \dots, x_n) \otimes v_k = 
\sum_{\ell} d_L \omega_\ell (x_1, \dots, x_n) \otimes u_\ell ,
\end{equation}
where $\sum_\ell \omega_\ell \otimes u_\ell$ is an image of the restriction of  
$\Omega$ to $\bigwedge^{n-1} (L\otimes 1)$, under the canonical isomorphism 
$\C^{n-1}(L, M\otimes V) \simeq \C^{n-1}(L,M) \otimes V$.

The equality (\ref{yoyo13}) implies that some nontrivial linear 
combination of $\varphi_k$'s is a coboundary, hence $\overline\xi$ is injective,
and the statement of the lemma follows.
\end{proof}

As simple as it is, it is curious to note that this is essentially the same
reasoning as in \cite[Lemma]{invar-forms} about cohomology of generalized 
de Rham complex.

Generally, a full computation of cohomology $\Homol^*(L\otimes A, M\otimes V)$,
except for some special cases in low degrees, is a hopeless task, and this 
embedding is very far from being surjective (for cohomology of degree $2$, see 
\cite[Proposition 3.1]{low}).

\begin{lemma}\label{hom}
Let $L$ be a Lie algebra, $A$ an associative commutative algebra with unit, 
$M$ an $L$-module, and $V_1$, $V_2$ two unital $A$-modules. Suppose that
either $M$, or both $V_1$, $V_2$ are finite-dimensional. Then
\begin{multline}\label{eq-isom}
\Hom_{L\otimes A}(M \otimes V_1, M \otimes V_2) \\ \simeq
\set{\varphi \in \End_K(M)}{\varphi(LM) = 0; L\varphi(M) = 0} 
\otimes \Hom_K(V_1,V_2) + \End_L(M) \otimes \Hom_A(V_1,V_2).
\end{multline}
\end{lemma}

At this point we start to use repeatedly linear-algebraic reasonings similar to
those in \cite{low} and \cite{without-unit}. This simple case concerns maps of 
one argument, and thus does not involve any symmetrization, unlike the other 
cases considered further, in subsequent sections.

\begin{proof}
Let $\Phi: M\otimes V_1 \to M \otimes V_2$ be a homomorphism of 
$L\otimes A$-modules. Because of the fi\-ni\-te-di\-men\-si\-o\-na\-li\-ty
condition,
\begin{equation*}
\Hom_K(M \otimes V_1, M \otimes V_2) \simeq \End_K(M) \otimes \Hom_K(V_1, V_2) ,
\end{equation*}
so we may write 
\begin{equation}\label{sum}
\Phi = \sum_{i\in I} \varphi_i \otimes \alpha_i
\end{equation}
for some finite set $I$, and some $\varphi_i \in \End_K(M)$, 
$\alpha_i \in \Hom_K(V_1,V_2)$. The condition of $\Phi$ to be a homomorphism of
$L\otimes A$-modules can be then written as
\begin{equation}\label{eq}
\sum_{i\in I} \varphi_i(x \bullet m) \otimes \alpha_i(a \bullet_1 v) -
\sum_{i\in I} (x \bullet \varphi_i(m)) \otimes (a \bullet_2 \alpha_i(v)) = 0
\end{equation}
for any $x\in L$, $a\in A$, $m\in M$, $v\in V_1$, where 
$\bullet_1$ and $\bullet_2$ denote the $A$-action on $V_1$ and $V_2$, 
respectively. Substitute in this equality $a=1$:
$$
\sum_{i\in I} (\varphi_i(x \bullet m)  - x \bullet \varphi_i(m)) \otimes \alpha_i(v) = 0 .
$$
Hence, we may assume that $\varphi_i(x \bullet m) = x \bullet \varphi_i(m)$ holds
for all $\varphi_i$'s and any $x\in L$, $m\in M$; in other words, each 
$\varphi_i$'s is an endomorphism of $M$ as an $L$-module. Substituting this back
to (\ref{eq}), we have:
$$
\sum_{i\in I} \varphi_i(x \bullet m) \otimes 
(\alpha_i(a \bullet_1 v) - a \bullet_2 \alpha_i(v)) = 0 .
$$
Hence there is a partition of the set of indices $I = I_1 \cup I_2$ such that
$\varphi_i(LM) = 0$ for $i\in I_1$, and 
$\alpha_i(a \bullet_1 v) = a \bullet_2 \alpha_i(v)$ for $i\in I_2$,
and the statement of the lemma follows.
\end{proof}

It seems to be difficult to derive similarly general results about 
isomorphisms between current modules over a current Lie algebra, as it is not 
clear how the condition of bijectivity of the map (\ref{sum}) could be related 
to conditions imposed on the summands $\varphi_i$'s and $\alpha_i$'s.
However, in the particular case where the number of homomorphisms between
the respective modules is rather ``small'', we can use the previous lemma to establish

\begin{theorem}\label{th}
Let $L$ be a Lie algebra having a cohomologically nontrivial finite-dimensional 
module $M$ such that either $LM = M$, or $M$ does not contain a trivial 
submodule, and $A$ an associative commutative algebra with unit having
an infinite set $\mathcal V$ of non-isomorphic unital modules of a fixed finite
dimension. Then in each of the following cases:
\begin{enumerate}
\item 
$\dim \End_L(M) = 1$,
\item
for any two non-isomorphic modules $V_1$, $V_2\in \mathcal V$, $\dim \Hom_A (V_1, V_2) \le 1$,
\end{enumerate}
the Lie algebra $L\otimes A$ has an infinite number of cohomologically 
nontrivial non-isomorphic modules of a fixed finite dimension.
\end{theorem}

\begin{proof}
Consider the set of $L\otimes A$-modules of the form $M\otimes V$, where 
$V\in \mathcal V$. By Lemma \ref{embed}, all these modules are cohomologically 
nontrivial. Due to the condition imposed on $M$, the first summand in the 
right-hand side of (\ref{eq-isom}) vanishes, hence by Lemma \ref{hom}, in both 
cases (i) and (ii), each $L\otimes A$-homomorphism between two modules 
$M\otimes V_1$ and $M\otimes V_2$ can be represented as a decomposable map
$\varphi \otimes \alpha$, where $\varphi\in \End_L(M)$ and 
$\alpha \in \Hom_A(V_1,V_2)$.
Such a map is bijective (i.e. is an isomorphism of $L\otimes A$-modules)
if and only if both tensor components $\varphi$ and $\alpha$ are bijective.
Hence two $L\otimes A$-modules $M\otimes V_1$ and $M\otimes V_2$ are 
isomorphic if and only if the $A$-modules $V_1$ and $V_2$ are isomorphic, 
whence the conclusion of the theorem.
\end{proof}

Thus, to exhibit an example of a finite-dimensional Lie algebra having an 
infinite number of cohomologically nontrivial non-isomorphic modules of a fixed
finite dimension (and thus answer Question 1 from \cite{dzhu-al} in the case of
characteristic zero, and provide a counterexample to the claim of Theorem 1 
from \cite{dzhu-al} in the case of positive characteristic), it is enough to 
exhibit algebras and modules satisfying the conditions of Theorem \ref{th}.

Take, for example, a Lie algebra $L$ having a nontrivial, irreducible, 
cohomologically nontrivial module $M$: in characteristic zero, according to 
a converse to the Whitehead theorem (\cite[Theorem]{converse}), we may take
any Lie algebra which is not a direct sum of a semisimple algebra and a nilpotent 
algebra; in positive characteristic, by combination of Dzhumadil'daev's 
results -- that any finite-dimensional Lie algebra has a 
finite-dimensional cohomologically nontrivial module 
(\cite[\S 1, Corollary 2]{dzhu-sbornik}; also established independently in 
\cite[Corollary 2.2]{farn-s}), and that triviality of any irreducible 
cohomologically nontrivial module implies nilpotency 
(\cite[\S 4, Theorem]{dzhu-sbornik}) -- we may take any nonnilpotent Lie 
algebra.

Over an algebraically closed base field, the condition (i) of Theorem \ref{th}
is satisfied then due to Schur's lemma, so in this case it is enough to 
pick any finite-dimensional associative commutative algebra $A$ with an infinite
number of modules of a fixed finite dimension -- for example, any algebra of 
infinite representation type. Over an arbitrary (infinite) field, however, we exhibit
a concrete example of an algebra $A$ having an infinite number of modules
satisfying the property (ii).

\begin{example}
Let $A$ be a $3$-dimensional local algebra whose radical is 
square-zero: that is, $A$ has a basis $\{1, x, y\}$, with multiplication
$x^2 = xy = yx = y^2 = 0$. This is one of the simplest examples of algebras
of infinite representation type. Finite-dimensional indecomposable 
representations of $A$ are known: to our knowledge, they were described first
in \cite[Proposition 5]{heller-reiner}. (A more general problem about canonical 
forms of a pair of mutually annihilating nilpotent -- and not necessary 
square-zero -- matrices was solved in the famous paper \cite{gelfand-ponomarev};
see also \cite{nrsb}). From them, it is possible to pick an infinite family $\mathcal V$ satisfying the property 
(ii) of Theorem \ref{th}. This can be done in multiple ways, and one of the 
simplest is the following. 

Consider a family $V_t$ of unital $2$-dimensional $A$-modules, parametrized by 
element of the base field $t$, defined as follows:
$V_t$ has a basis $\{u,v\}$ with the following $A$-action:
$$
x \bullet u = v, \quad y \bullet u = tv, \quad x \bullet v = 0, 
\quad y \bullet v = 0 .
$$
Straightforward computation show that for different $t,s\in K$, the modules 
$V_t$ and $V_s$ are non-isomorphic, and that $\Hom_A (V_t, V_s)$ is 
$1$-dimensional, linearly spanned by the map $u \mapsto v$, $v \mapsto 0$.
\end{example}

Note that for any finite-dimensional Lie algebra of positive characteristic, 
the number of \emph{irreducible} cohomologically nontrivial finite-dimensional
modules is finite, by the same arguments used in the proof of 
\cite[Theorem 1]{dzhu-al}. In view of this, perhaps the Question 1 from 
\cite{dzhu-al} should be reformulated as follows:

\begin{question}
Do there exist finite-dimensional Lie algebras over a field of characteristic
zero having an infinite number of cohomologically nontrivial 
\textit{irreducible} non-isomorphic modules of a fixed finite dimension?
\end{question}

We conjecture that the answer to this question is negative, and a possible way
to prove this is to employ the standard ultraproduct considerations, allowing
to deduce the zero characteristic case from the cases of positive 
characteristics.

\begin{question}
What would be a ``noncommutative'' version of Theorem \ref{th}, where current
Lie algebras are replaced by Lie algebras of the form $\mathsf{sl}_n(A)$ or 
$\mathsf{gl}_n(A)$ for an associative (and not necessary commutative) algebra 
$A$?
\end{question}

We expect that in this way one may to obtain further interesting connections
between cohomology of Lie algebras and representation theory of associative 
algebras.

\section{Poisson structures on current and Kac-Moody algebras}\label{sec-poisson-current}

In this section, the characteristic of the base field $K$ is assumed to be 
as big as needed, in particular, to allow all denominators appearing in the 
formulas. Thus, in the results about general current Lie algebras below
(Theorem \ref{poisson-current}, Corollary \ref{cent-perfect}, 
Lemmas \ref{cent-current}--\ref{forms}), we assume the characteristic is 
$\ne 2,3$, while when dealing with simple Lie algebras and, in particular, 
with formulas like (\ref{jordan}), the characteristic is assumed to be zero.
 
The classical example of a vector space of functions on a manifold equipped,
from one hand, with operation of point-wise function multiplication, and, 
from another hand, with Poisson bracket, leads to the abstract notion of a 
Poisson algebra: a vector space $A$ with two algebra structures: one, Lie, 
denoted by $\liebrack$, and another, associative commutative, denoted by $\star$
with compatibility condition saying that Lie multiplication by each element is a 
derivation of a commutative algebra structure: 
\begin{equation}\label{p}
[z,x \star y] = [z,x] \star y + x \star [z,y]
\end{equation}
for any $x,y,z\in A$.

It is only natural to consider a more general situation, when condition of 
commutativity of $\star$ is dropped, arriving at so-called 
\textit{noncommutative Poisson algebras}. Such algebras were studied in many 
papers (\cite{kubo-kac-moody} is just one of them). As in \cite{kubo-kac-moody},
we adopt a ``Lie-centric'' (as opposed to ``associative-centric'') viewpoint,
according to which one fixes a Lie algebra $L$ and considers all possible 
multiplications $\star$ on the vector space $L$ such that 
$(L, \liebrack, \star)$ forms a Poisson algebra.

We go, however, a bit further, and drop associativity condition of $\star$, 
thus retaining only the compatibility condition (\ref{p}) (and, of course, 
Lieness of $\liebrack$). So, given a Lie algebra $L$, we call a \textit{Poisson structure on $L$}
any algebra structure $\star$ on the vector space $L$ such that (\ref{p}) holds.
If $\star$ satisfies additional conditions like commutativity,
associativity, etc., we will speak about commutative, associative, etc., 
Poisson structures on $L$.

Note that a Poisson structure on a Lie algebra $L$ is nothing but an $L$-module 
homomorphism from $L\otimes L$ to $L$. As such, the set of all Poisson structures on $L$ forms
a vector space under usual operations of additions and multiplication on scalar
of linear maps from $L\otimes L$ to $L$. 

There is a canonical decomposition of an $L$-module $L\otimes L$ into the 
direct sum of commutative and anti-commutative components: 
$L\otimes L = (L\wedge L) \>\oplus\> (L\vee L)$.
Accordingly, the space of Poisson structures on $L$ decomposes into the direct 
sum of two spaces, consisting of commutative and anti-commutative Poisson 
structures on $L$.

Obviously, for every Lie algebra $L$, a bilinear map $L \times L \to L$ 
proportional to the Lie multiplication, is a Poisson structure on $L$ 
(what amounts to the Jacobi identity). Of course, such Poisson structures
are not very interesting, and one wishes to consider all Poisson structures
on $L$ modulo them. A bit more generally, we may wish to eliminate
Poisson structures which can be written in the form $x\star y = \omega ([x,y])$
for some linear map $\omega: L \to L$. Then (\ref{p}) together with the Jacobi
identity implies that 
\begin{equation}\label{omega}
[z,\omega([x,y])] = \omega([z,[x,y]])
\end{equation}
for any $x,y,z\in L$. Note that this is nothing but the \textit{centroid} 
(i.e., the set of linear maps on an algebra commuting with all
multiplications) of the commutant $[L,L]$. 
Moreover, when $L$ is simple and the base field is algebraically closed, the
centroid $\Cent(L)$ of $L$ coincides with multiplications on the elements of 
the base field (\cite[Chapter 10, \S 1, Theorem 1]{jacobson}), so we arrive in 
this case at Poisson structures which are proportional to Lie multiplications 
we started with. 

Centroid also appears in a different way in our context: if $\star$ is a 
Poisson structure on a Lie algebra $L$ with a nonzero center $\Z(L)$, then 
(\ref{p}) implies that for any element $z\in \Z(L)$, the left and right 
Poisson multiplication on $z$, i.e. the maps $x \mapsto z \star x$ and 
$x \mapsto x \star z$, are elements of the centroid of $L$.

Poisson structures of type (\ref{omega}) will be called \textit{trivial},
and the quotient of the space of all Poisson structures on a Lie algebra $L$
by the trivial ones, will be denoted by $\mathscr P(L)$.
Poisson structures, whose representatives form a basis of $\mathscr P(L)$,
will be called \textit{basic Poisson structures}.
Note that trivial Poisson structures are always anti-commutative,
so any nonzero commutative Poisson structure is nontrivial.

Somewhat similarly, on a skew-symmetrization $A^{(-)}$ of an associative 
algebra $A$, one always have a commutative Poisson structure defined by the
``Jordan'' multiplication: $x \circ y = \frac{1}{2} (xy + yx)$ for $x,y\in A$. 
Though such ``Jordan'' Poisson structures are defined only for Lie algebras 
which are skew-symmetrizations of associative ones, for the Lie algebra 
$\mathsf{sl}_n(K)$ we can define a 
related Poisson structure by adding a term ``correcting'' for tracelessness:
\begin{equation}\label{jordan}
X \star Y = \frac{1}{2} (XY + YX) - \frac{1}{n} \Tr(XY)E, 
\end{equation}
where $X,Y$ are traceless matrices of order $n$, and $E$ is the identity 
matrix. It is easy to check that this is indeed a Poisson structure on 
$\mathsf{sl}_n(K)$ (what amounts to the fact that $(X,Y) \mapsto \Tr(XY)$ is a 
symmetric invariant form on $\mathsf{sl}_n(K)$), and it will be called a 
\textit{standard commutative Poisson structure}.
Note that this is no longer a Jordan structure. The case of $\mathsf{sl}_2(K)$ 
is degenerate, as multiplication defined by (\ref{jordan}) vanishes.

Our goal is to describe Poisson structures on current Lie algebras $L\otimes A$
in terms of $L$ and $A$. We do this under mild technical assumptions.
The Poissonity condition (\ref{p}) is similar to the cocycle equation, and the 
same technique used to compute low-degree cohomology of $L\otimes A$, applies.

As the centroid, essentially, is ``almost'' isomorphic to the space of trivial 
Poisson structures, we first determine the centroid of current Lie algebras. 
This is generally known from the literature, perhaps in a slightly less general
or explicit form.

\begin{lemma}\label{cent-current}
Let $L$ be a Lie algebra, $A$ an associative commutative algebra with unit,
and one of $L$, $A$ is finite-dimensional. Then 
$$
\Cent(L \otimes A) \simeq \Cent(L) \otimes A + 
\Hom_K(L/[L,L], (\Z(L) + [L,L])/[L,L]) \otimes \End_K(A).
$$
Any element of $\Cent (L\otimes A)$ can be written as the sum of decomposable 
maps $\varphi \otimes \alpha$, $\varphi\in \End_K(L)$, $\alpha\in \End_K(A)$,
of one of the following form:
\begin{enumerate}
\item 
$\varphi\in \Cent(L)$ and $\alpha = \R_u$ for some $u\in A$;
\item
$\varphi(L) \subseteq \Z(L)$ and $\varphi([L,L]) = 0$.
\end{enumerate}
\end{lemma}

This generalizes \cite[Lemma 5.1]{krylyuk}, and overlaps with 
\cite[Corollary 2.23]{benkart-neher} and \cite[Remark 2.19.1]{gundogan}.

\begin{proof}
Let $\Phi\in \Cent(L\otimes A)$. Due to the finite-dimensionality assumption, 
we may write $\Phi = \sum_{i\in I} \varphi_i \otimes \alpha_i$ for suitable 
linear maps $\varphi_i\in \End_K(L)$ and $\alpha_i\in \End_K(A)$. The 
centroidity condition then reads
\begin{equation}\label{cent}
\sum_{i\in I} 
\varphi_i([x,y]) \otimes \alpha_i(ab) - [x,\varphi_i(y)]\otimes a\alpha_i(b) = 0
\end{equation}
for any $x,y\in L$, $a,b\in A$. Substituting here $a=1$, we get
$$
\sum_{i\in I} (\varphi_i([x,y]) - [x,\varphi_i(y)]) \otimes \alpha_i(b) = 0 ,
$$
and hence we may assume that all $\varphi_i$'s belong to $\Cent(L)$. Taking 
this into account, (\ref{cent}) can be rewritten as
\begin{equation*}
\sum_{i\in I} \varphi_i([x,y]) \otimes (\alpha_i(ab) - a\alpha_i(b)) = 0.
\end{equation*}
Hence there is a partition of the set of indices $I = I_1 \cup I_2$ such that
$\varphi_i([x,y]) = [x,\varphi_i(y)] = 0$, $x,y\in L$ for $i\in I_1$, and
$\alpha_i(ab) = a\alpha_i(b)$, $a,b\in A$ for $i\in I_2$. 
The latter condition is equivalent to $\alpha_i(a) = au_i$ for some $u_i\in A$.
\end{proof}

If $L$ is perfect, i.e. $[L,L] = L$, then the second summand in the isomorphism
of Lemma \ref{cent-current} disappears:

\begin{corollary}\label{cent-perfect}
Let $L$ be a perfect Lie algebra, $A$ an associative commutative algebra with 
unit, and one of $L$, $A$ is finite-dimensional. Then 
$\Cent (L\otimes A) \simeq \Cent(L) \otimes A$.
\end{corollary}

After this warm-up, let us turn to computation of Poisson structures on current
and close to them Lie algebras.
Let us call a Poisson structure $\star$ on a Lie algebra $L$ \textit{left skew}
if it satisfies the identity $[x,y] \star z = [y,z] \star x$, and \textit{right skew}
if it satisfies the identity $x \star [y,z] = y \star [z,x]$.

\begin{theorem}\label{poisson-current}
Let $L$ be a perfect Lie algebra not having nonzero left skew and right skew 
Poisson structures, $A$ an associative commutative algebra with unit, and one 
of $L$, $A$ is finite-dimensional. Then 
$\mathscr P (L\otimes A) \simeq \mathscr P(L)\otimes A$, and basic Poisson 
structures on $L\otimes A$ can be chosen as
\begin{equation}\label{maltese}
(x \otimes a) \>\bigstar\> (y \otimes b) = (x \star y ) \otimes abu ,
\end{equation}
where $x,y\in L$, $a,b\in A$, for some Poisson structure $\star$ on $L$, 
and $u\in A$.
\end{theorem}

\begin{proof}
Let $\Phi = \sum_{i\in I} \varphi_i \otimes \alpha_i$, where 
$\varphi_i:L \times L \to L$, $\alpha_i: A \times A \to A$ are bilinear maps,
be a Poisson structure on $L\otimes A$. Writing the Poissonity condition
(\ref{p}) for triple $x\otimes a$, $y\otimes b$, $z\otimes c$, we get:
\begin{equation}\label{yoyo1}
\sum_{i\in I} 
[z, \varphi_i(x,y)] \otimes c\alpha_i(a,b) - 
\varphi_i([z,x],y) \otimes \alpha_i(ca,b) - 
\varphi_i(x,[z,y]) \otimes \alpha_i(a,cb) = 0.
\end{equation}
Substitute here $c=1$:
$$
\sum_{i\in I} ([z, \varphi_i(x,y)] - \varphi_i([z,x],y) - \varphi_i(x,[z,y])) \otimes
\alpha_i(a,b) = 0.
$$
Hence, replacing $\varphi_i$'s by their appropriate linear combinations, we get that
each $\varphi_i$ satisfies 
$[z, \varphi_i(x,y)] = \varphi_i([z,x],y) + \varphi_i(x,[z,y])$ and hence is
a Poisson structure on $L$. Substituting this back to (\ref{yoyo1}), we get:
\begin{equation}\label{yoyo5}
\sum_{i\in I} 
\varphi_i([z,x],y) \otimes (c \alpha_i(a,b) - \alpha_i(ca,b)) + 
\varphi_i(x,[z,y]) \otimes (c \alpha_i(a,b) - \alpha_i(a,cb)) = 0.
\end{equation}
Symmetrizing this with respect to $x,z$ and $y,z$, we get respectively:
\begin{equation*}
\sum_{i\in I} 
(\varphi_i(x,[z,y]) + \varphi_i(z,[x,y]) \otimes (c \alpha_i(a,b) - \alpha_i(a,cb)) = 0
\end{equation*}
and
\begin{equation*}
\sum_{i\in I} 
(\varphi_i([z,x],y) + \varphi_i([y,x],z)) \otimes (c \alpha_i(a,b) - \alpha_i(ca,b)) = 0.
\end{equation*}
The vanishing of the first tensor factors here is equivalent for $\varphi_i$ to
be right and left skew respectively, and hence they cannot vanish for a nonzero
$\varphi_i$. Hence for each $i\in I$,  $c\alpha_i(a,b) - \alpha_i(a,cb) = 0$ and
$c\alpha_i(a,b) - \alpha_i(ca,b) = 0$, what implies 
$\alpha_i(ab) = ab\alpha_i(1)$. It is easy to check that each 
$\varphi_i \otimes \alpha_i$ for such $\alpha_i$'s is a Poisson structure on 
$L\otimes A$.

So, we see that each Poisson structure on $L\otimes A$ can be written as a sum
of Poisson structures of the form (\ref{maltese}), and hence the space
of all Poisson structures on $L\otimes A$ is isomorphic to the tensor product 
of the space of all Poisson structures on $L$ with $A$. To prove the asserted 
isomorphism, note that perfectness of $L$ implies perfectness of $L\otimes A$, 
and hence the space of trivial Poisson structures on $L\otimes A$ is isomorphic
to the centroid of $L\otimes A$, and by Corollary \ref{cent-perfect}, to 
$\Cent(L)\otimes A$. But $\Cent(L)$ is precisely the space of trivial Poisson 
structures on $L$, so factoring out the isomorphism of the spaces of Poisson 
structures established above by the spaces of trivial Poisson structures, we 
get the desired isomorphism.
\end{proof}

Though it seems that it is impossible to get a definitive result describing
Poisson structures on $L\otimes A$ in terms of $L$ and $A$ in the most general 
case, by using much more elaborate (and cumbersome) arguments along these lines,
one may significantly relax restrictions on $L$. That way, nonperfectness of 
$L$ becomes responsible for Poisson structures on $L\otimes A$ of the form 
$\psi \otimes \alpha$ with conditions
like $\psi([L,L],L) = 0$, skew Poisson structures $\varphi$ on $L$
become responsible for Poisson structures on $L\otimes A$ of the form 
$\varphi \otimes \alpha$, where $\alpha$ is composed from various generalized 
derivations on $A$, etc. 

\begin{corollary}\label{poisson-classical}
Let $\mathfrak g$ be a finite-dimensional simple Lie algebra and $A$ an 
associative commutative algebra with unit, defined over an algebraically closed
field $K$ of characteristic $0$. Then 
$\mathscr P(\mathfrak g \otimes A) \simeq A$ in the case 
$\mathfrak g = \mathsf{sl}_n(K)$, $n\ge 3$, and vanishes in all other 
cases. In the former case the basic Poisson structures can be chosen as
$$
(X \otimes a) \star (Y \otimes b) = 
\big( \frac{1}{2}(XY + YX) - \frac{1}{n}\Tr(XY)E \big)
\otimes abu ,
$$
where $X,Y\in \mathsf{sl}_n(K)$, $a,b\in A$, for some $u\in A$.
\end{corollary}

Note that in the case $A = K$, this corollary reduces to 
\cite[Lemma 3.1]{benkart-osborn}, which can be interpreted as a description of 
Poisson structures on simple finite-dimensional Lie algebras (also rediscovered
in \cite[Lemma 8]{lecomte-roger} and \cite{kubo-late}). In fact, we base on that result. It is also 
curious to note that exactly the same dichotomy between 
$\mathsf{sl}_n(K)$, $n\ge 3$, and the other simple types occurs in the second 
and third degree (co)homology 
$\Homol_2(\mathfrak g \otimes A, \mathfrak g \otimes A)$ and 
$\Homol_3(\mathfrak g \otimes A, K)$ (see \cite[Propositions 1 and 2]{C}; also 
follows from a careful substitution of $\mathfrak g$ into the isomorphism 
displayed at p.~94 in the published version, and p.~17 in the arXiv version of 
\cite{low}).

\begin{proof}[Proof of Corollary \ref{poisson-classical}]
To be able to apply Theorem \ref{poisson-current}, we should demonstrate that 
$\mathfrak g$ does not have nonzero left and right skew Poisson structures. 
This is not difficult: consider, for example, a left skew Poisson structure 
$\star$ on $\mathfrak g$. By the above-mentioned 
\cite[Lemma 3.1]{benkart-osborn}, for the case $\mathfrak g = \mathsf{sl}_n(K)$,
$n\ge 3$, we have 
$$
X \star Y = \frac{\lambda}{2} (XY + YX) - \frac{\lambda}{n} \Tr(XY) E + 
\mu(XY - YX)
$$ 
for some $\lambda, \mu \in K$. Taking $X$ and $Z$ to be the 
block-diagonal matrices with 
$\left(\begin{matrix} 1 & 0 & 0 \\ 0 & -1 & 0 \\ 0 & 0 & 0 \end{matrix}\right)$ and
$\left(\begin{matrix} 0 & 0 & 1 \\ 0 &  0 & 0 \\ 0 & 0 & 0 \end{matrix}\right)$
in the upper-left corner, and zero everywhere else, respectively 
(these are just a semisimple element and a corresponding root vector in 
the subalgebra isomorphic to $\mathsf{sl}_3(K)$), we get 
$[X,X] \star Z = 0$ and $[X,Z] \star X = Z \star X$ is equal to the block-diagonal 
matrix with 
$\left(\begin{matrix} 0 & 0 & \frac{\lambda}{2} - \mu \\ 0 & 0 & 0 \\ 0 & 0 & 0 \end{matrix}\right)$
in the upper-left corner, whence $\lambda - 2\mu = 0$.

Taking now $X^\prime$ to be the block-diagonal matrix with
$\left(\begin{matrix} 1 & 0 & 0 \\ 0 & 1 & 0 \\ 0 & 0 & -1 \end{matrix}\right)$
in the upper-left corner, and zero everywhere else, we get
$[X^\prime,X^\prime] \star Z = 0$ and $[X^\prime,Z] \star X^\prime = 3Z \star X^\prime$ 
is equal to the block-diagonal matrix with 
$\left(\begin{matrix} 0 & 0 & -\frac{3}{2} \lambda - 9\mu \\ 0 & 0 & 0 \\ 0 & 0 & 0 \end{matrix}\right)$
in the upper-left corner, whence $\lambda + 6\mu = 0$.

Thus $\lambda = \mu = 0$ and $\star$ vanishes.

For the all other types of $\mathfrak g$, $\star$ is proportional to the Lie 
multiplication on $\mathfrak g$, so the left skewness amounts to the identity $[[x,y],z] = [[y,z],x]$, which obviously does not hold in a simple Lie algebra.

The case of right skew Poisson structures is similar. 

A consecutive application of Theorem \ref{poisson-current} and 
\cite[Lemma 3.1]{benkart-osborn} concludes the proof.
\end{proof}

In order to apply this further to Kac--Moody algebras, we need to consider 
semidirect sums of a current Lie algebra with an algebra of derivations of $A$,
and their central extensions. Consider a Lie algebra defined as the vector space
$(L \otimes A) \oplus Kz \oplus \mathcal D$, where $\mathcal D$ is a nontrivial
Lie algebra of derivations of $A$, with the following multiplication:
\begin{align}\label{eq-ext}
[x\otimes a, y\otimes b] &= [x,y] \otimes ab + \langle x, y \rangle \xi(a,b) z  \\
[x \otimes a, d] &= x \otimes d(a)\notag
\end{align}
for $x,y \in L$, $a,b \in A$, $d\in \mathcal D$, and $z$ is a central element. 
Here $\langle \cdot, \cdot \rangle$ is a nonzero symmetric bilinear 
invariant form on $L$, i.e., a symmetric bilinear map satisfying the equality
$\langle [x,y],z \rangle + \langle y,[x,z]\rangle = 0$ for any $x,y,z\in L$;
and $\xi$ is a nonzero $\mathcal D$-invariant element of 
$\HC^1(A)$, the first-degree cyclic cohomology of $A$, i.e., a skew-symmetric 
bilinear map $\xi: A \times A \to K$ satisfying the following conditions:
$\xi(ab,c) + \xi(ca,b) + \xi(bc,a) = 0$, and $\xi(d(a),b) + \xi(a,d(b)) = 0$
for any $a,b,c\in A$, $d\in \mathcal D$. Note that the former condition implies
$\xi(1,A) = 0$.

Such algebra is a nonsplit central extension of the semidirect sum 
$(L \otimes A) \inplus \mathcal D$.
For brevity, we will call algebras with multiplication (\ref{eq-ext})
\emph{extended current Lie algebras}.
Specializing this construction to the case where $K$ is an algebraically closed
field of characteristic zero, $L = \mathfrak g$, a simple finite-dimensional 
Lie algebra, $\langle \cdot, \cdot \rangle$ is the Killing form on 
$\mathfrak g$, $A = K[t,t^{-1}]$, the algebra of Laurent polynomials, 
$\mathcal D = Kt\frac{d}{dt}$, and $\xi (f,g) = Res(g \frac{df}{dt})$ for 
$f,g\in K[t,t^{-1}]$, we get non-twisted affine Kac-Moody algebras 
(see \cite[Chapter 7]{kac}).

In the next series of lemmas and corollaries to Theorem \ref{poisson-current}, 
we gradually approach to computation of Poisson structures on some class of 
extended current Lie algebras.

\begin{lemma}\label{lemma-centr}
Let $\mathcal L$ be an extended current Lie algebra such that $L$ is a perfect 
centerless Lie algebra, and one of $L$, $A$ is finite-dimensional. Then any 
element in $\Cent(\mathcal L)$ is of the form
\begin{align}\label{eq-centr-ext}
x \otimes a &\mapsto x \otimes au \notag\\
d &\mapsto ud + \varphi(d)z       \\
z &\mapsto \lambda z              \notag
\end{align}
where $x\in L$, $a\in A$, $d\in \mathcal D$, for some $u\in A$ such that
$\mathcal D(u) = 0$, a linear map $\varphi: \mathcal D \to K$ such that
$\varphi([\mathcal D,\mathcal D]) = 0$, and $\lambda\in K$ such that
$\xi(au,b) = \lambda \xi(a,b)$ for any $a,b\in A$.
\end{lemma}

\begin{proof}
Let $\Phi \in \Cent(\mathcal L)$. Note the the center of $\mathcal L$ coincides
with $Kz$. It is straightforward that $\Phi$ preserves the center of 
$\mathcal L$, and hence induces the map on the quotient $\mathcal L/Kz$,
which belongs to the centroid of 
$\mathcal L/Kz \simeq (L \otimes A) \inplus \mathcal D$.

Let us determine the structure of the latter first. Let 
$\Psi\in \Cent ((L\otimes A) \inplus \mathcal D)$. Then
$$
\Psi(L\otimes A) = \Psi([L,L]\otimes A) = [L\otimes A, \Psi(L\otimes A)] 
\subseteq L\otimes A .
$$
Consequently $\Psi$, being restricted to $L\otimes A$, belongs to 
$\Cent(L\otimes A)$, and by Corollary \ref{cent-perfect}, 
$\Psi(x\otimes a) = \psi(x) \otimes au$ for some $\psi\in \Cent(L)$ and 
$u\in A$. The centroidity condition 
$[\Psi(x\otimes a),d] = \Psi([x\otimes a,d])$, for any 
$x\in L$, $a\in A$, $d\in \mathcal D$, implies that $\mathcal D(u) = 0$.
Next, the centroidity condition 
$[x\otimes a, \Psi(d)] = \Psi([x\otimes a, d])$ is equivalent to
\begin{equation}\label{yoyo9}
[x\otimes a, \Psi(d)] = \psi(x) \otimes d(a)u.
\end{equation}
Writing the $L\otimes A$-valued component of the restriction of $\Psi$ to 
$\mathcal D$ in the form $\sum_{i\in I} x_i \otimes \alpha_i$
for some $x_i\in L$ and linear maps $\alpha_i: \mathcal D\to A$, and 
substituting in (\ref{yoyo9}) $a=1$, we get
$$
\sum_{i\in I} [x,x_i] \otimes \alpha_i(d) = 0.
$$
Thus all $x_i$'s belong to the center of $L$ and hence vanish. This shows that
$\Psi(\mathcal D) \subseteq \mathcal D$, and the condition (\ref{yoyo9}) can be
rewritten as 
$$
x \otimes \Psi(d)(a) = \psi(x) \otimes d(a)u
$$
for any $x\in L$, $a\in A$, $d\in \mathcal D$. This implies 
$\psi(x) = \alpha x$ and $\Psi(d) = \alpha ud$ for some $\alpha\in K$.

Consequently, $\Cent ((L\otimes A) \inplus \mathcal D) \simeq A^{\mathcal D}$, 
and any element in that centroid is of the form 
$$
x\otimes a + d \mapsto x\otimes au + ud ,
$$ 
where $x\in L$, $a\in A$, $d\in \mathcal D$, for some $u\in A$ such that 
$\mathcal D(u) = 0$.

Returning to the algebra $\mathcal L$, we may now write $\Phi$ in the form
\begin{align*}
\Phi(x \otimes a) &= x \otimes au + \tau(x \otimes a) z  \\
\Phi(d) &= ud + \varphi(d)z                              \\
\Phi(z) &= \lambda z                                     
\end{align*}
for some $\tau\in \Hom_K(L\otimes A,K)$, $\varphi\in \Hom_K(\mathcal D,K)$,
$\lambda\in K$, and $u\in A$ as above. Then the centroidity condition
$\Phi([x\otimes a, y\otimes b]) = [\Phi(x\otimes a),y\otimes b]$ is equivalent 
to
$$
\tau([x,y]\otimes ab) = \langle x,y\rangle (\xi(au,b) - \lambda \xi(a,b))
$$
for any $x,y\in L$, $a,b\in A$. The left-hand side here is skew-symmetric in
$x,y$, while the right-hand side is symmetric, therefore they vanish separately,
what implies the vanishing of $\tau$, and the identity 
$\xi(au,b) = \lambda \xi(a,b)$.

The centroidity condition $\Phi([d,f]) = [\Phi(d),f]$ for any 
$d,f\in \mathcal D$ is equivalent to $\varphi([d,f]) = 0$.

It is easy to check that the map defined by (\ref{eq-centr-ext}) subject the
all specified conditions, indeed belongs to the centroid of $\mathcal L$.
\end{proof}

We will need also the following statement about bilinear invariant forms
on current Lie algebras.

\begin{lemma}\label{forms}
Let $L$ be a Lie algebra, $A$ an associative commutative algebra with unit,
and one of $L$, $A$ is finite-dimensional. Then each bilinear invariant form on
$L\otimes A$ can be represented as the sum
of decomposable forms $\varphi\otimes \alpha$, $\varphi: L\times L \to K$, 
$\alpha: A\times A\to K$ satisfying one of the two following conditions:
\begin{enumerate}
\item 
$\varphi ([x,y],z) = \varphi([z,x],y)$ for any $x,y,z\in L$, and 
$\alpha(a,b) = \beta(ab)$ for some linear map $\beta: A\to K$;
\item $\varphi ([L,L],L) = 0$.
\end{enumerate}
\end{lemma}

\begin{proof}
The proof is absolutely similar to those of Theorem \ref{poisson-current}, and, 
in fact, is given in \cite[Theorem 2]{without-unit}. The condition of symmetricity of bilinear
forms used there at the very end and does not affect the main body of the proof.
Specializing to the case where $A$ contains a unit, we get the statement of
the lemma.
\end{proof}

\begin{corollary}\label{cor-yoyo}
Let $A$ be an associative commutative algebra with unit, and $\mathcal D$ an 
abelian Lie algebra of derivations of $A$. Then
$$
\mathscr P((\mathfrak g\otimes A) \inplus \mathcal D) \simeq 
\begin{cases} 
U(\mathcal D,A) \oplus U(\mathcal D,A) \oplus 
\set{u\in A}{\mathcal D(u) = \mathcal D(A)u = 0}, \\
              \qquad\qquad\qquad\qquad\qquad\qquad\qquad\mathfrak g = \mathsf{sl}_n(K), n\ge 3 \\
U(\mathcal D,A) \oplus U(\mathcal D,A), \text{ otherwise},
\end{cases}
$$
where 
$U(\mathcal D,A) = 
\set{\beta\in \Hom_K(\mathcal D,A)}{\mathcal D(\beta(\mathcal D)) = 0}$.
The basic Poisson structures can be chosen as
follows (assuming $x,y\in \mathfrak g$, $a,b\in A$, $d,d^\prime\in \mathcal D$):
\begin{enumerate}
\item 
$(x\otimes a) \star (y\otimes b) = (\frac 12 (xy + yx) - \frac 1n \Tr(xy)E) \otimes abu$,
where $u\in A$ such that $\mathcal D(u) = 0$ and $\mathcal D(A)u = 0$;

\noindent
$\mathcal D \star (\mathfrak g\otimes A) 
= (\mathfrak g\otimes A) \star \mathcal D = \mathcal D \star \mathcal D = 0$; 

\noindent 
these Poisson structures exist in the case $\mathfrak g = \mathsf{sl}_n(K)$, 
$n\ge 3$ only; 

\item
$d \star (x\otimes a) = x\otimes a\beta(d)$, 
$d \star d^\prime = \beta(d)d^\prime$, where $\beta\in U(\mathcal D,A)$; 

\noindent
$(\mathfrak g\otimes A) \star (\mathfrak g\otimes A) 
= (\mathfrak g\otimes A) \star \mathcal D = 0$;

\item
$(x\otimes a) \star d = x\otimes a\gamma(d)$, 
$d \star d^\prime = \gamma(d^\prime)d$, where $\gamma\in U(\mathcal D,A)$;

\noindent
$(\mathfrak g\otimes A) \star (\mathfrak g\otimes A) 
= \mathcal D \star (\mathfrak g\otimes A) = 0$.
\end{enumerate}
\end{corollary}

\begin{proof}
Each Poisson structure on $(\mathfrak g\otimes A) \inplus \mathcal D$, being 
restricted to $(\mathfrak g\otimes A) \times (\mathfrak g\otimes A)$, can be 
decomposed into the sum of two maps $\Phi$ and $\Psi$ with values in 
$\mathfrak g\otimes A$ and $\mathcal D$, respectively. Writing the Poissonity condition for elements of $\mathfrak g\otimes A$,
we get
\begin{multline}\label{phi-psi}
[z\otimes c, \Phi(x\otimes a, y\otimes b)] + [z\otimes c, \Psi(x\otimes a, y\otimes b)]
\\ 
= \Phi ([z,x]\otimes ca, y\otimes b) + \Phi (x, [z,y]\otimes ca)
\end{multline}
and
\begin{equation*}
\Psi ([z,x]\otimes ca, y\otimes b) + \Psi (x, [z,y]\otimes ca) = 0
\end{equation*}
for any $x,y,z\in \mathfrak g$, $a,b,c,\in A$.
Hence, $\Psi$ is a $\mathcal D$-valued invariant bilinear form on 
$\mathfrak g\otimes A$. Now apply Lemma \ref{forms}. As $\mathfrak g$ is 
perfect, forms of type (ii) vanish. Bilinear maps on $\mathfrak g$ satisfying 
the condition in (i) can be decomposed into the sum of symmetric and 
skew-symmetric ones. Symmetric maps coincide with symmetric invariant forms on 
$\mathfrak g$ and hence are proportional to the Killing form 
$\langle \cdot,\cdot \rangle$, and skew-symmetric maps vanish by 
\cite[Lemma 2]{without-unit}. Hence $\Psi$ can be written in the form 
$\Psi (x\otimes a, y\otimes b) = \langle x, y \rangle \beta (ab)$ for some
linear map $\beta: A \to \mathcal D$.

Now we may proceed as in the proof of Theorem \ref{poisson-current}.
Writing $\Phi$ in the form $\sum_{i\in I}\varphi_i \otimes \alpha_i$ and 
substituting this into (\ref{phi-psi}), we get the equality (\ref{yoyo1}) with
\begin{equation}\label{term}
\langle x, y \rangle z\otimes \beta (ab)(c)
\end{equation}
instead of $0$ at the right-hand side. This expression vanishes under the
substitution $c=1$, and we may proceed exactly as in the
 proof of Theorem \ref{poisson-current} to arrive at the equality (\ref{yoyo5}) with
the same term (\ref{term}) instead of $0$ at the right-hand side. Setting in 
the latter $x = y = z$, we get the vanishing left-hand side, and 
$\langle x, x \rangle x\otimes \beta (ab)(c)$ at the right-hand side. Now 
picking $x\in \mathfrak g$ such that $\langle x, x \rangle \ne 0$, we get
$\beta = 0$. Hence $\Psi = 0$, $\Phi$ is a Poisson structure on 
$\mathfrak g \otimes A$, and by Corollary \ref{poisson-classical},
$$
\Phi (x\otimes a, y\otimes b) = (x \star y) \otimes abu + [x,y]\otimes abv
$$
in the case of $\mathfrak g = \mathsf{sl}_n(K)$, $n \ge 3$, where $\star$ is 
the standard commutative Poisson structure on $\mathsf{sl}_n(K)$, and 
$u,v\in A$, and $\Phi (x\otimes a, y\otimes b) = [x,y]\otimes abv$ for the all
other types of $\mathfrak g$.

Now writing the Poissonity condition for triple $y\otimes b, d, x\otimes a$, 
where $x,y\in \mathfrak g$, $a,b\in A$, $d\in \mathcal D$, we get that 
$d(u) = d(v) = 0$ for any $d\in \mathcal D$.

Now denoting, by abuse of notation, the whole Poisson structure on 
$(\mathfrak g\otimes A) \inplus \mathcal D$ by the same letter $\Phi$, let us 
see what happens when one of the arguments of $\Phi$ lies in $\mathcal D$. 
Writing the Poissonity condition for triple $y\otimes b, d, x\otimes a$, we get
\begin{equation}\label{yoyo7}
[y\otimes b, \Phi (d, x\otimes a)] = \Phi (y\otimes d(b), x\otimes a) + 
\Phi (d, [y,x]\otimes ba) ,
\end{equation}
what implies 
$\Phi (\mathcal D, \mathfrak g\otimes A) 
= \Phi (\mathcal D, [\mathfrak g, \mathfrak g] \otimes A)
\subseteq \mathfrak g\otimes A$. We may write
$$
\Phi (d, x\otimes a) = \sum_{i\in I} \varphi_i(x) \otimes \alpha_i(d,a)
$$
for suitable (bi)linear maps $\varphi_i: \mathfrak g \to \mathfrak g$ and 
$\alpha_i: \mathcal D \times A \to A$. Substituting this into (\ref{yoyo7}) and
setting there $b=1$, we get
\begin{equation*}
\sum_{i\in I} (\varphi_i([y,x]) - [y,\varphi_i(x)]) \otimes \alpha_i(d,a) = 0.
\end{equation*}
Thus we may assume that $\varphi_i([y,x]) - [y,\varphi_i(x)] = 0$ for each $i\in I$,
i.e. each $\varphi_i$ belongs to the centroid of $\mathfrak g$ and hence is a 
multiplication by an element of the base field. But then we have 
$\Phi (d, x\otimes a) = x\otimes \alpha(d,a)$ for a suitable bilinear map
$\mathcal D \times A \to A$ and all $x\in \mathfrak g$, $a\in A$, 
$d\in \mathcal D$, and the Poissonity condition (\ref{yoyo7}) reduces to
$$
[y,x] \otimes (b\alpha(d,a) - \alpha(d,ba)) = \Phi(y\otimes d(b), x\otimes a)
$$
for any $x,y\in \mathfrak g$, $a,b\in A$, $d\in \mathcal D$.
If $\mathfrak g = \mathsf{sl}_n(K)$, $n \ge 3$, then setting in this equality 
$x=y$, we get $(x \star x) \otimes d(b)au = 0$, and picking $x$ such that 
$x \star x \ne 0$ (for example, $x = \diag (1,1,-2,0, \dots, 0)$), we get that 
$\mathcal D(A)u = 0$, and the whole equality reduces to
\begin{equation}\label{yoyo8}
b\alpha(d,a) - \alpha(d,ba) = ad(b)v.
\end{equation}

If $\mathfrak g$ is different from $\mathsf{sl}_n(K)$, $n \ge 3$, then the term
with the standard commutative Poisson structure is absent in $\Phi$, and we also
get the equality (\ref{yoyo8}). Setting in the latter equality $a=1$, we see
that it is equivalent to the condition $\alpha(d,a) = a\beta(d) - d(a)v$
for some linear map $\beta: \mathcal D \to A$. The Poissonity condition for 
triple $d, d^\prime, x\otimes a$ for $x\in \mathfrak g$, $a\in A$, 
$d,d^\prime \in \mathcal D$, together with the condition $\mathcal D(v) = 0$ and
abelianity of $\mathcal D$, implies then that 
$\mathcal D(\beta(\mathcal D)) = 0$.

Quite similarly, we have that 
$$
\Phi (x\otimes a, d) = a\gamma(d) + d(a)v
$$ 
for any $x\in \mathfrak g$, $a \in A$, $d\in \mathcal D$, and a suitable linear
map $\gamma: \mathcal D \to A$ such that $\mathcal D(\gamma(\mathcal D)) = 0$.

Finally, the Poissonity condition written for three elements of $\mathcal D$ 
implies that $[\mathcal D,\Phi(\mathcal D, \mathcal D)] = 0$ and hence 
$\Phi(\mathcal D, \mathcal D) \subseteq \mathcal D$, and the Poissonity 
condition for triple $x\otimes a, d, d^\prime$ implies 
$$
\Phi (d, d^\prime)(a) = d(a)\gamma(d^\prime) + d^\prime(a)\beta(d)
$$
for any $a\in A, d,d^\prime\in \mathcal D$.

Summarizing all this together, it is easy to single out the trivial Poisson structures:
those are precisely the maps involving the element $v$. Writing the generic form
of the Poisson structure modulo these maps, we get the desired statement.
\end{proof}

\begin{corollary}\label{cor-yoyo1}
Let $\mathfrak L$ be an extended current Lie algebra over an algebraically 
closed base field $K$ of characteristic zero, $L$ is simple finite-dimensional 
Lie algebra, $\mathcal D = Kd$ is one-dimensional, $\dim A > 2$, and 
$A^{\mathcal D} = K$. Then the basic Poisson structures on $\mathcal L$ can be 
chosen from the following list:
\begin{enumerate}
\item 
$(x \otimes a) \star (y \otimes b) = \langle x,y \rangle f(ab) z$,

\noindent
$z \star d = -d \star z = z$,

\noindent
$(\mathfrak g \otimes A) \star d = d \star (\mathfrak g \otimes A) = 
d \star d = 0$,

\noindent
where $f:A\to K$ is a linear map such that $f(ad(b)) = \xi(a,b)$ for any 
$a,b\in A$;

\item
$(x \otimes a) \star (y \otimes b) = \langle x,y \rangle f(ab) z$,

\noindent
$d \star \mathcal L = \mathcal L \star d = 0$,

\noindent
where $f:A\to K$ is a linear map such that $f(ad(b)) = 0$ for any $a,b\in A$;

\item
$d \star (x\otimes a) = x\otimes a$,

\noindent
$d \star d = d$,

\noindent
$(\mathfrak g \otimes A) \star (\mathfrak g \otimes A) 
= (\mathfrak g \otimes A) \star d = z \star d = d \star z = 0$;

\item
$(x\otimes a) \star d = x\otimes a$,

\noindent
$d \star d = d$,

\noindent
$(\mathfrak g \otimes A) \star (\mathfrak g \otimes A) 
= d \star (\mathfrak g \otimes A) = z \star d = d \star z = 0$;

\item
$d \star d = z$,

\noindent
$(\mathfrak g \otimes A) \star \mathcal L 
= \mathcal L \star (\mathfrak g \otimes A) = z \star d = d \star z = 0$.

\end{enumerate}
Additionally, for all of them 
$(\mathfrak g \otimes A) \star z = z \star (\mathfrak g \otimes A) = z \star z 
= 0$.
\end{corollary}

\begin{proof}
If $\Phi$ is a Poisson structure on $\mathcal L$, then 
$\Phi(z,\cdot)$ and $\Phi(\cdot,z)$ are elements of the centroid of 
$\mathcal L$. According to Lemma \ref{lemma-centr}, $\Cent(\mathcal L)$ is 
$2$-dimensional, linearly spanned by the identity map, and the map sending
$d$ to $z$, and $(\mathfrak g\otimes A) \oplus Kz$ to zero. Hence we may write
\begin{align}\label{eq-d}
\Phi(x \otimes a, z) &= \Phi(z, x \otimes a) = \alpha x \otimes a  \notag\\  
\Phi(d,z) &= \alpha d + \beta z                                    \\
\Phi(z,d) &= \alpha d + \gamma z  \notag \\
\Phi(z,z) &= \alpha z             \notag
\end{align}
where $x\in \mathfrak g$, $a\in A$, for some $\alpha, \beta, \gamma \in K$.

Restrict $\Phi$ to 
$((\mathfrak g \otimes A) \inplus Kd) \times ((\mathfrak g \otimes A) \inplus Kd)$, 
and decompose the restriction as the sum of a 
$((\mathfrak g \otimes A) \inplus Kd)$-valued bilinear map $\Psi$, and a 
$Kz$-valued bilinear map. The Poissonity condition, taking into account 
equalities (\ref{eq-d}), then implies
\begin{multline}\label{eq-first}
[t\otimes c,\Psi(x\otimes a, y\otimes b)] \\=
\Psi([t,x]\otimes ca, y\otimes b) + \Psi(x\otimes a, [t,y]\otimes cb) 
+ \alpha(\langle t,x \rangle \xi(c,a) y\otimes b 
     + \langle t,y \rangle \xi(c,b) x\otimes a)
\end{multline}
and
\begin{gather}
[d,\Psi(x\otimes a, y\otimes b)] 
= \Psi([d,x\otimes a],y\otimes b) + \Psi(x\otimes a,[d,y\otimes b]) \notag
\\
[x\otimes a,\Psi(d,y\otimes b)] 
= \Psi([x\otimes a,d],y\otimes b) + \Psi(d,[x,y] \otimes ab) 
+ \alpha\langle x,y \rangle \xi(a,b) d     \notag
\\
[x\otimes a,\Psi(y\otimes b,d)] 
= \Psi([x,y]\otimes ab,d) + \Psi(y\otimes b,[x\otimes a,d])
+ \alpha\langle x,y \rangle \xi(a,b) d     \notag
\\
[d,\Psi(d,x\otimes a)] = \Psi(d,[d,x\otimes a]) \label{eq-yoyo2} \\
[d,\Psi(x\otimes a,d)] = \Psi([d,x\otimes a],d) \notag \\
[x\otimes a,\Psi(d,d)] = \Psi([x\otimes a,d],d) + \Psi(d,[x\otimes a,d]) \notag
\\
[d,\Psi(d,d)] = 0 \notag
\end{gather}
for any $x,y,t\in \mathfrak g$, $a,b,c\in A$. 

Decompose further $\Psi$ into the $(\mathfrak g\otimes A)$-valued and 
$Kd$-valued components, and proceed as in the proof of Corollary \ref{cor-yoyo}.
By the same reasoning, the $Kd$-valued component, being a $Kd$-valued invariant
form on $\mathfrak g \otimes A$, can be written as 
$\langle x,y \rangle \otimes \beta(ab)d$ for some linear map $\beta: A \to K$.
Substituting this into equality (\ref{eq-first}), we get the Poissonity 
condition for the $\mathfrak g \otimes A$-valued component, up to the term
$$
\langle x,y \rangle t \otimes \beta(ab)d(c) 
+ 
\alpha(\langle t,x \rangle \xi(c,a) y\otimes b 
     + \langle t,y \rangle \xi(c,b) x\otimes a)
$$
instead of (\ref{term}) at the right-hand side. Proceeding exactly as in the 
proof of Corollary \ref{cor-yoyo}, we get that this term vanishes whenever
$x=y=t$, i.e. 
$$
\alpha(\xi(c,a)b + \xi(c,b)a) = \beta(ab)d(c)
$$
for any $a,b,c\in A$. An easy calculation yields
$\beta(a)d(b) + \beta(b)d(a) = 0$ for any $a,b\in A$. If $\beta$ is nonzero,
this implies that $d$ vanishes on $\Ker \beta$, whence $\Ker \beta = K$ and 
$\dim A = 2$, a contradiction. Hence $\beta$ vanishes, $\alpha = 0$, and
equalities (\ref{eq-first})--(\ref{eq-yoyo2}) imply that $\Psi$ is a 
Poisson structure on $(\mathfrak g \otimes A) \inplus \mathcal D$.

Now look at Corollary \ref{cor-yoyo}. Since $\mathcal D$ is one-dimensional, and
$A^{\mathcal D} = K$, the space $U(\mathcal D,A)$ is one-dimensional, linearly 
spanned by the map sending $d$ to $1$. The conditions on $u$ in the structures
of type (i) imply the vanishing of structures of that type. Hence the space of
Poisson structures on $(\mathfrak g \otimes A) \inplus Kd$ is linearly spanned
by the Lie bracket on that algebra, and the following two bilinear maps:
\begin{align*}
(d,x\otimes a) &\mapsto x\otimes a  \\
(x\otimes a,d) &\mapsto 0 
\end{align*}
and
\begin{align*}
(d,x\otimes a) &\mapsto 0  \\
(x\otimes a,d) &\mapsto x\otimes a ,
\end{align*}
both of them sending $(d,d)$ to $d$, and 
$(\mathfrak g \otimes A) \times (\mathfrak g \otimes A)$ to zero.

Going back to $\Phi$, we may now write a generic Poisson structure on 
$\mathcal L$, modulo the maps proportional to the Lie bracket on $\mathcal L$, 
as:
\begin{align*}
\Phi(x\otimes a, y\otimes b) &= \Theta(x\otimes a,y\otimes b)z \\
\Phi(d,x\otimes a) &= \lambda x \otimes a + \Gamma(x\otimes a)z \\
\Phi(x\otimes a,d) &= \mu x \otimes a + \Gamma^\prime(x\otimes a)z \\
\Phi(d,d) &= (\lambda + \mu)d + \eta z
\end{align*}
for some (bi)linear maps
$\Theta: (\mathfrak g \otimes A) \times (\mathfrak g \otimes A) \to K$,
$\Gamma,\Gamma^\prime: \mathfrak g\otimes A \to K$, and 
$\lambda,\mu,\eta \in K$. Taking into account (\ref{eq-d}) 
(with $\alpha = 0$), the Poissonity condition for $\Phi$ now amounts to the 
following equalities:
\begin{gather}
\Theta([t,x] \otimes ac, y\otimes b) + \Theta(x\otimes a, [t,y] \otimes bc) = 0
\label{eq-yoyo-first}  \\
\Theta(x \otimes d(a), y \otimes b) + \Theta( x \otimes a, y \otimes d(b)) = 0
\label{eq-yoyo-second} \\
\Theta(x \otimes d(a), y\otimes b) + \Gamma([x,y] \otimes ab) 
+ \beta \langle x,y \rangle \xi(a,b) = 0    \label{eq-yoyo-3}
\\
\Theta(y\otimes b, x\otimes d(a)) + \Gamma^\prime([x,y]\otimes ab)
+ \gamma \langle x,y \rangle \xi(a,b) = 0   \label{eq-yoyo-4}
\end{gather}
for any $x,y,t\in \mathfrak g$, $a,b,c\in A$. Substituting into 
(\ref{eq-yoyo-3}) $a=1$, we get that $\Gamma$ vanishes, similarly for 
$\Gamma^\prime$. 
The equality (\ref{eq-yoyo-first}) is the condition of invariance of the 
bilinear form $\Theta$. Referring to Lemma \ref{forms} and 
\cite[Lemma 2]{without-unit} in the same way as in the proof of 
Corollary \ref{cor-yoyo}, we have 
$\Theta(x \otimes a, y \otimes b) = \langle x,y \rangle f(ab)$ for some
linear map $f: A \to K$.

The rest is straightforward: (\ref{eq-yoyo-second}) now amounts to 
$f(d(A)) = 0$, and (\ref{eq-yoyo-3}) and (\ref{eq-yoyo-4}), to
$f(d(a)b) + \beta \xi(a,b) = 0$ and $f(bd(a)) - \gamma \xi(a,b) = 0$, 
respectively. Hence $\gamma = -\beta$, and the statement of the corollary 
follows.
\end{proof}

\begin{corollary}
The space of nontrivial Poisson structures on an affine non-twisted Kac-Moody 
algebra is $3$-dimensional, spanned by structures defined in 
{\rm (iii)}, {\rm (iv)} and {\rm (v)} of Corollary {\rm \ref{cor-yoyo1}}.
\end{corollary}

\begin{proof}
It is easy to see that the map $f$ satisfying the conditions in cases (i) and 
(ii) of Corollary \ref{cor-yoyo1}, vanishes in the Kac-Moody case. 
\end{proof}

This generalizes the result from \cite{kubo-kac-moody}, where it is proved that
any \emph{associative} Poisson structure on an affine Kac-Moody algebra is
a linear combination of the specified $3$ structures. The proof in 
\cite{kubo-kac-moody} goes by lengthy case-by-case computations with the 
corresponding root systems. As any of the structures of types (iii), 
(iv) and (v) is associative, it turns out that any Poisson structure on an 
affine Kac--Moody algebra is decomposed into the sum of an associative 
structure, and a trivial structure proportional to the Lie bracket.

The same approach can be used to describe Poisson structures in many other 
cases which involve tensor products in that or another way: twisted Kac--Moody 
algebras, toroidal Lie algebras, semisimple Lie algebras over a field of 
positive characteristic, variations of Kantor-Koecher-Tits construction,
Lie algebras graded by root systems, etc. Some of these computations will be 
technically (much) more cumbersome, but all of them should follow the same 
scheme. Yet another similar computation will be presented in the next section.

\section{Poisson structures on $\mathsf{sl}_n(A)$}\label{sec-poisson-gl}

We may try to consider a noncommutative analog of the problem in the previous 
section. Namely, for unital associative algebras $A$ and $B$, whether
one can describe Poisson structures on the Lie algebra $(A\otimes B)^{(-)}$
in terms of $A$ and $B$, similar to those of Theorem \ref{poisson-current}? 
It seems that in general the answer is negative. Let, for example, both 
$A$ and $B$ be a semidirect sum of $K1$ and a nilpotent associative algebra $N$
of degree $3$ (i.e., $N^3 = 0$). It is not difficult
to see that then $(A\otimes B)^{(-)}$ is an abelian Lie algebra.
Poisson structures on an abelian Lie algebra $L$ coincide with the whole space of 
bilinear maps $L \times L \to L$, so, in that case Poisson structures
on $(A\otimes B)^{(-)}$ and on $A^{(-)}$ and $B^{(-)}$ appear to be not
related at all.

This example suggests that to be able to get such a description,
algebras $A$ and $B$, or at least one of them, should be ``far from nilpotent''.
We are indeed able to do so in the particular (and arguably one of the most 
interesting) case when $B$ is, in a sense, as much not nilpotent as possible -- 
that is, is isomorphic to the full matrix algebra. In this case the Lie algebra
$(A\otimes B)^{(-)}$ is isomorphic to the Lie algebra $\mathsf{gl}_n(A)$. 
Actually, we will even consider not that algebra itself, 
but its close (and more interesting) cousin $\mathsf{sl}_n(A)$. One of the 
reasons to do so is that in the case of $\mathsf{gl}_n(A)$ the picture is 
obscured by a lot of ``degenerate'', not very interesting, and having a 
cumbersome formulation Poisson structures related to the fact that the center of
$\mathsf{gl}_n(K)$ is not zero.

Thus, we will get, in a sense, a ``noncommutative'' version of 
Corollary \ref{poisson-classical}.

Recall that for an associative algebra $A$, the Lie algebra $\mathsf{sl}_n(A)$ 
is defined as subalgebra of $\mathsf{gl}_n(A)$ consisting of matrices over $A$ 
with trace lying in $[A,A]$. There is a split extension of Lie algebras 
$$
0 \to \mathsf{sl}_n(A) \to \mathsf{gl}_n(A) \to A/[A,A] \to 0 ,
$$
and, as a vector space, $\mathsf{sl}_n(A)$ can be identified with 
$$
(\mathsf{sl}_n(K) \otimes A) \oplus (E \otimes [A,A]) \subseteq 
\mathsf{gl}_n(K) \otimes A \simeq \mathsf{gl}_n(A)
$$ 
subject to multiplication
\begin{multline*}
[X \otimes a, Y \otimes b] = XY \otimes ab - YX \otimes ba \\ = 
\Big( XY - \frac 1n \Tr(XY)E \Big) \otimes ab -
\Big( YX - \frac 1n \Tr(XY)E \Big) \otimes ba + \frac 1n \Tr(XY)E \otimes [a,b].
\end{multline*}
Note that the first tensor factors in the first two terms here are Poisson 
structures on $\mathsf{sl}_n(K)$.

Throughout this section, the characteristic of the base field $K$ is assumed
to be zero. This is to allow denominators in the formula above providing 
realization of $\mathsf{sl}_n(K)$ we prefer to work with. A more careful 
analysis would allow to prove essentially the same result assuming the 
characteristic $\ne 2,3$, but we will not take this pain.

\begin{theorem}\label{sln}
Let $A$ be an associative algebra with unit. The basic Poisson structures on 
$\mathsf{sl}_n(A)$, $n\ge 3$, can be chosen as follows:
\begin{enumerate}

\item 
$(X\otimes a) \star (Y\otimes b) = \frac 12 (XY\otimes abu + YX\otimes bau) 
+ \Tr(XY)E\otimes \gamma(ab)$ for either $X\in \mathsf{sl}_n(K)$, $a\in A$, or 
$X=E, a\in [A,A]$, the same for $Y$,$b$,
where $u\in Z(A)$ and $\gamma: A \to \Z(A)$ is a linear map such that 
$\gamma([A,A]) = \gamma([[A,A],A]A) = 0$, and 
$\frac 1n au + \gamma(a) \in [A,A]$ for any $a\in A$;

\item
$(E\otimes a) \star (X\otimes b) = X\otimes \alpha(a)b$ for $a\in [A,A]$ and 
either $X\in \mathsf{sl}_n(K)$, $b\in A$, or $X=E$, $b\in [A,A]$, where 
$\alpha: [A,A] \to \Z(A)$ is a linear map such that $\alpha([[A,A],[A,A]]) = 0$;

\noindent $(\mathsf{sl}_n(K)\otimes A) \star \mathsf{sl}_n(A) = 0$.

\item
$(X\otimes a) \star (E\otimes b) = X\otimes \beta(b)a$ for $b\in [A,A]$ and 
either $X\in \mathsf{sl}_n(K)$, $a\in A$, or $X=E$, $a\in [A,A]$, where 
$\beta: [A,A] \to \Z(A)$ is a linear map such that $\beta([[A,A],[A,A]]) = 0$;

\noindent $\mathsf{sl}_n(A) \star (\mathsf{sl}_n(K)\otimes A) = 0$.

\item
$(E\otimes a) \star (E\otimes b) = E\otimes \delta(a,b)$ for $a,b\in [A,A]$, 
where $\delta: [A,A] \times [A,A] \to \Z(A) \cap [A,A]$ is a bilinear map such 
that $\delta([c,a],b) + \delta(a,[c,b])= 0$ for any $a,b,c\in [A,A]$;

\noindent 
$(\mathsf{sl}_n(K)\otimes A) \star \mathsf{sl}_n(A) = 
\mathsf{sl}_n(A) \star (\mathsf{sl}_n(K) \otimes A) = 0$.

\end{enumerate}
\end{theorem}

Here $\Z(A)$ denotes, as usual, the center of $A$.

\begin{proof}
Let $\Phi$ be a Poisson structure on $\mathsf{sl}_n(A)$. Identifying 
$\mathsf{sl}_n(A)$ with a subspace of $\mathsf{gl}_n(K) \otimes A$ as described 
above, extend $\Phi$ in an arbitrary way to the whole 
$(\mathsf{gl}_n(K) \otimes A) \times (\mathsf{gl}_n(K) \otimes A)$ (say, by 
picking a complementary subspace and letting $\Phi$ vanish on it).

Let 
\begin{equation}\label{phi}
\Phi = \sum_{i\in I} \varphi_i \otimes \alpha_i ,
\end{equation}
where $\varphi_i$'s are bilinear maps on $\mathsf{gl}_n(K)$, and $\alpha_i$'s
are bilinear maps on $A$. Writing the Poissonity condition for triple 
$X\otimes a, Y\otimes b, Z\otimes c \in \mathsf{sl}_n(A)$ (that is, either 
$X\in \mathsf{sl}_n(K)$ and $a\in A$, or $X = E$ and $a\in [A,A]$, similarly
for $Y\otimes b$ and $Z\otimes c$), we get:
\begin{align}
\sum_{i\in I} 
  Z\varphi_i(X,Y) &\otimes c\alpha_i(a,b) - \varphi_i(X,Y)Z \otimes \alpha_i(a,b)c  \notag \\
- \varphi_i(ZX,Y) &\otimes \alpha_i(ca,b) + \varphi_i(XZ,Y) \otimes \alpha_i(ac,b)  \label{yoyo2} \\
- \varphi_i(X,ZY) &\otimes \alpha_i(a,cb) + \varphi_i(X,YZ) \otimes \alpha_i(a,bc) = 0.  \notag
\end{align}
Like in the proof of Theorem \ref{poisson-current}, substituting here $c=1$, we can
assume that each $\varphi_i$ satisfies the Poissonity condition 
\begin{equation}\label{yoyo6}
[Z, \varphi_i (X,Y)] = \varphi_i([Z,X],Y) + \varphi_i(X,[Z,Y])
\end{equation}
for any $X,Y\in \mathsf{gl}_n(K), Z\in \mathsf{sl}_n(K)$. 

Each $\varphi_i$, being restricted to 
$\mathsf{sl}_n(K) \times \mathsf{sl}_n(K)$, can be decomposed into the sum of 
two linear maps with values in $\mathsf{sl}_n(K)$ and $KE$, respectively. The 
$\mathsf{sl}_n(K)$-valued summand in this decomposition is a Poisson structure 
on $\mathsf{sl}_n(K)$, and the $KE$-valued one is an invariant bilinear form on
$\mathsf{sl}_n(K)$. The latter is known to be proportional to the Killing form 
(see, for example, \cite[Chap. 1, \S 6, Exercises 7(b) and 18(a,b)]{bourbaki}), 
and hence is proportional to $\Tr(XY)$. Then by 
\cite[Lemma 3.1]{benkart-osborn}, each $\varphi_i$, being restricted to 
$\mathsf{sl}_n(K) \times \mathsf{sl}_n(K)$, belongs to the vector spaces of 
bilinear maps generated by $(X,Y) \mapsto XY$, $(X,Y) \mapsto YX$, and 
$(X,Y) \mapsto \Tr(XY)E$. Hence $\Phi$, being restricted to 
$(\mathsf{sl}_n(K)\otimes A) \times (\mathsf{sl}_n(k)\otimes A)$, can be 
written in the form
\begin{equation}\label{yoyo4}
\Phi (X\otimes a, Y\otimes b) = XY\otimes \alpha(a,b) + YX\otimes \beta(a,b) +
\Tr(XY)E \otimes \gamma(a,b)
\end{equation}
for some bilinear maps $\alpha, \beta, \gamma: A \times A \to A$.
Then the Poissonity condition for triple 
$X\otimes a$, $Y\otimes b$, $Z\otimes c$, $X,Y,Z\in \mathsf{sl}_n(K)$, 
$a,b,c\in A$, reads 
\begin{align}
  & XYZ \otimes (\alpha(a,bc) - \alpha(a,b)c)   \notag        \\
+ & XZY \otimes (\alpha(ac,b) - \alpha(a,cb))   \notag        \\
+ & YXZ \otimes (\beta(ac,b) - \beta(a,b)c)     \notag        \\
+ & YZX \otimes (\beta(a,bc) - \beta(ca,b))     \notag        \\  
+ & ZXY \otimes (c\alpha(a,b) - \alpha(ca,b))   \label{yoyo3} \\
+ & ZYX \otimes (c\beta(a,b) - \beta(a,cb))     \notag        \\  
+ & \Tr(XY)Z \otimes [c,\gamma(a,b)]            \notag        \\
+ & E \otimes \Big(
\Tr(XYZ) (\gamma(a,bc) - \gamma(ca,b)) + \Tr(XZY) (\gamma(ac,b) - \gamma(a,cb))
\Big) 
= 0.                                            \notag
\end{align}

Setting in (\ref{yoyo3}) $X=Y=Z$, we get:
\begin{multline*}
X^3 \otimes 
\Big( 
[c,\alpha(a,b)] - \alpha([c,a],b) - \alpha(a,[c,b]) +
[c,\beta(a,b)] - \beta([c,a],b) - \beta(a,[c,b])
\Big) \\
+ \Tr(X^2)X \otimes [c,\gamma(a,b)]                  
- \Tr(X^3)E \otimes \Big( \gamma([c,a],b) + \gamma(a,[c,b]) \Big)
= 0,
\end{multline*}
for any $X\in \mathsf{sl}_n(K)$ and $a,b,c\in A$. Taking here, for example, 
$X = \diag (1,1,-2,0, \dots, 0)$, we see that each of the summands vanishes 
separately. In particular, the second tensor factor vanishes, and, so is the
corresponding term in (\ref{yoyo3}) (those containing $\Tr(XY)Z$). Moreover,
taking into account the vanishing of the third tensor factor, the last term in 
(\ref{yoyo3}) can be rewritten as
$$
(\Tr(XYZ) - \Tr(XZY))E \otimes (\gamma(a,bc) - \gamma(ca,b)) .
$$

Now take $X, Y, Z$ as follows: all their elements are zero besides the
$3 \times 3$ upper left corner, and the latter is
$\left(\begin{matrix} 0 & 1 & 1 \\ 0 & 0 & 0 \\ 0 & 0 & 0 \end{matrix}\right)$ for $X$,
$\left(\begin{matrix} 0 & 0 & 0 \\ 1 & 0 & 1 \\ 0 & 0 & 0 \end{matrix}\right)$ for $Y$,
and
$\left(\begin{matrix} 0 & 0 & 0 \\ 0 & 0 & 0 \\ 1 & 1 & 0 \end{matrix}\right)$ for $Z$
(so, actually we are dealing with $3 \times 3$ matrices)\footnote[2]{
It seems that almost any other $3$ matrices will do. See 
Problem \ref{ques-matr} at the end of this section.}.
It is straightforward to check that the $6$ matrices formed by all triple 
products of $X,Y,Z$ are linearly independent, and $\Tr(XYZ) - \Tr(XZY) = 0$. 
This implies that the second tensor factors in each of the first $6$ terms of 
(\ref{yoyo3}) vanish. But then (\ref{yoyo3}) reduces to just the last term, and
picking any $3$ matrices such that $\Tr(XYZ) - \Tr(XZY) \ne 0$ shows that the 
second tensor factor in the last term vanishes too.
By elementary manipulations with these vanishing conditions, we have
$\alpha(a,b) = abu$ for some $u\in \Z(A)$, and $\beta(a,b) = bav$ for some 
$v\in \Z(A)$. 

Now consider the equality 
$$
\Tr(XYZ)(\gamma(a,bc) - \gamma(ca,b)) + \Tr(XZY)(\gamma(ac,b) - \gamma(a,cb)) 
= 0 .
$$
Taking here, for example, $X = Z = \diag (1,-1,0, \dots, 0)$ and 
$Y = 
\diag(\left(\begin{matrix} 1 & 1 \\ 0 & -1 \end{matrix}\right), 0, \dots, 0)$,
we see that each of the two summands vanish separately, and, hence, 
$\gamma(a,b) = \gamma^\prime (ab)$ for a linear map $\gamma^\prime: A \to \Z(A)$
such that $\gamma^\prime ([A,A]) = 0$.

Clearly, $\mathsf{sl}_n(A)$ is closed under $\Phi$ if and only if 
$$
\frac 1n (abu + bav) + \gamma^\prime(ab) \in [A,A]
$$
for any $a,b\in A$, what can be rewritten as 
$\frac 1n a(u+v) + \gamma^\prime(a) \in [A,A]$ for any $a\in A$.

Now let us see what happens with values of $\Phi$ when one of the arguments 
belongs to $E\otimes [A,A]$. 
Substituting in the Poissonity condition for $\varphi_i$'s (\ref{yoyo6}) 
$X=E$ or $Y=E$, we get
that the map $\varphi_i(E,\cdot)$ or $\varphi_i(\cdot, E)$, respectively, 
commutes with $\ad Z$ for any $Z\in \mathsf{sl}_n(K)$. But then it commutes with
$\ad Z$ for any $Z\in \mathsf{gl}_n(K)$, i.e., belongs to the centroid of 
$\mathsf{gl}_n(K)$. Hence each of the $\varphi_i(E,X)$ and $\varphi_i(X,E)$ 
coincides with a scalar multiplication by $X$, and $\varphi_i(E,E)$ is 
proportional to $E$. This implies that the values of $\Phi$ for the arguments 
in question can be written in the form
\begin{align*}
\Phi(E\otimes a, X\otimes b) &= X \otimes \beta^\prime(a,b)  \\
\Phi(X\otimes b, E\otimes a) &= X \otimes \gamma^\prime(a,b) \\
\Phi(E\otimes a, E\otimes b) &= E \otimes \delta(a,b),
\end{align*}
for $X\in \mathsf{sl}_n(K)$ and $a\in [A,A], b\in A$ in the first two cases, 
and $a,b\in [A,A]$ in the third one, and where 
$\beta^\prime, \gamma^\prime, \delta: A \to A$ are some bilinear maps.

Now writing the Poissonity condition for triples 
$E\otimes a$, $X\otimes b$, $Y\otimes c$, $X,Y\in \mathsf{sl}_n(K)$, 
$a\in [A,A]$, $b,c\in A$, and 
$E\otimes a$, $X\otimes b$, $E\otimes c$, $X\in \mathsf{sl}_n(K)$, 
$a,c\in [A,A]$, $b\in A$, we get respectively:
\begin{multline*}
  XY \otimes (\beta^\prime(a,b)c - \beta^\prime(a,bc) + b[c,a]v)
+ YX \otimes (\beta^\prime(a,cb) - c\beta^\prime(a,b) + [c,a]bu)   \\
+ \Tr(XY)E \otimes \alpha^\prime([c,a]b) = 0
\end{multline*}
and
\begin{equation*}
X\otimes ([c,\beta^\prime(a,b)] - \beta^\prime([c,a],b) -  \beta^\prime(a,[c,b])) = 0.
\end{equation*}
Then, obviously, each of the summands in the first of these two equalities vanishes 
separately (take, for example, $X = Y = \diag (1,-1,0, \dots, 0)$, and 
$X = \diag (1,-1, 0, \dots, 0)$, 
$Y = 
\diag(\left(\begin{matrix} 0 & 1 \\ 0 & 0 \end{matrix}\right), 0, \dots, 0)$).
Consequently, $\alpha^\prime ([[A,A],A]A) = 0$, and elementary transformations 
with conditions on $\beta^\prime$ entail that 
$$
\beta^\prime(a,b) = abu + bav + \beta^{\prime\prime} (a)b
$$ 
for any $a\in [A,A]$, $b\in B$, where $\beta^{\prime\prime}: [A,A] \to \Z(A)$ 
is a linear map satisfying the condition 
$\beta^{\prime\prime} ([[A,A],[A,A]]) = 0$.

Similarly, the Poissonity condition for triples 
$X\otimes a$, $E\otimes b$, $Y\otimes c$ and 
$X\otimes a$, $E\otimes b$, $E\otimes c$ entails that
$$
\gamma^\prime(a,b) = abu + bav + \gamma^{\prime\prime} (b)a
$$
for any $a\in A$, $b\in [A,A]$, where $\gamma^{\prime\prime} : [A,A] \to \Z(A)$ 
is a linear map satisfying the condition 
$\gamma^{\prime\prime} ([[A,A],[A,A]]) = 0$.

At least, the Poissonity condition for triples 
$E\otimes a$, $E\otimes b$, $X\otimes c$, $a,b\in [A,A]$, $c\in A$, and
$E\otimes a$, $E\otimes b$, $E\otimes c$, $a,b,c\in [A,A]$ yields
$$
X \otimes ([c,\delta(a,b)] - \gamma^\prime([c,a],b) - \beta^\prime(a,[c,b])) = 0
$$
and
$$
E \otimes ([c,\delta(a,b)] - \delta([c,a],b) - \delta(a,[c,b])) = 0,
$$
what implies
$$
\delta(a,b) = abu + bav + \beta^{\prime\prime}(a)b + \gamma^{\prime\prime}(b)a 
+ \delta^\prime(a,b)
$$
for any $a,b\in [A,A]$, where $\delta^\prime: [A,A] \times [A,A] \to \Z(A)$ is 
a bilinear map such that $\delta^\prime([c,a],b) + \delta^\prime(a,[c,b]) = 0$ 
for any $a,b,c\in [A,A]$. For $\mathsf{sl}_n(A)$ to be closed under $\Phi$, 
$abu + bav + \delta^\prime(a,b)$ should lie in $[A,A]$, what is equivalent to 
$ab(u+v) + \delta^\prime (a,b) \in [A,A]$ for any $a,b \in [A,A]$.

It remains to collect all the obtained maps, to check in a straightforward way 
that they are indeed Poisson structures on $\mathsf{sl}_n(A)$ (in fact, we 
already did that in almost all the cases by considering various Poissonity 
conditions on basic elements of $\mathsf{sl}_n(A)$), and to take the quotient 
by trivial Poisson structures. As $\mathsf{sl}_n(A)$ is perfect, the trivial 
Poisson structures are of the form $(x,y) \mapsto \omega([x,y])$ where $\omega$
belongs to the centroid of $\mathsf{sl}_n(A)$. The latter can be computed in 
exactly the same way as the space of Poisson structures on $\mathsf{sl}_n(A)$ 
(though the computations are much simpler): the centroid of $\mathsf{sl}_n(A)$,
$n\ge 3$, is isomorphic to $\Z(A)$ and consists of the maps 
$$
(X\otimes a, Y\otimes b) \mapsto XY\otimes abu - YX\otimes bau
$$
for some $u\in \Z(A)$, where either $X\in \mathsf{sl}_n(K)$ and $a\in A$, or 
$X=E$ and $a\in [A,A]$, the same for $Y$ and $b$ 
(this is also noted without proof in \cite[Remark 5.6]{krylyuk}).
\end{proof}

\begin{remark}
Corollary \ref{poisson-classical} and Theorem \ref{sln} overlap in the case 
$\mathfrak g = \mathsf{sl}_n(K)$, $n\ge 3$, and $A$ commutative.
\end{remark}

\begin{remark}
One may treat the case of $\mathsf{sl}_2(A)$ along the same lines, but this
case, as it is often happens in such situations, is much more cumbersome.
Let us briefly outline the relevant reasonings, without going into details.
As every Poisson structure on $\mathsf{sl}_2(K)$ is trivial, we may assume in 
(\ref{yoyo3}) $\beta = -\alpha$. Using that, and the identity
$\Tr(XYZ) + \Tr(XZY) = 0$ which holds for any three $2 \times 2$ traceless 
matrices, one can easily single out in (\ref{yoyo3}), like in the general case, 
the term containing $\Tr(XY)Z$, but the remaining reasonings are cumbersome:
there are not ``enough'' elements in $\mathsf{sl}_2(K)$ to separate all the 
terms in (\ref{yoyo3}). Say, substituting into (\ref{yoyo3}) all possible 
values of the standard $\mathsf{sl}_2(K)$-basis
$\Big\{ 
\left(\begin{matrix} 1 & 0 \\ 0 & -1 \end{matrix}\right),
\left(\begin{matrix} 0 & 1 \\ 0 & 0 \end{matrix}\right),
\left(\begin{matrix} 0 & 0 \\ 1 & 0 \end{matrix}\right)
\Big\},
$
we get a (highly redundant) homogeneous linear system with rational coefficients
of $27$ equations in $7$ unknowns, which reduces to $4$ independent relations. 
Using these relations, one can derive that 
$$
\alpha(a,b) = u(ab+ba) + (ab+ba)u
$$ 
for some $u\in A$ satisfying some additional conditions, but the relation between $\gamma$ and $\alpha$ is more
involved.
\end{remark}

When dealing with the identity (\ref{yoyo3}), we were able to separate the terms
by picking concrete matrices $X$, $Y$, $Z$. Computer experiments suggest
that almost any $3$ traceless matrices will do.

\begin{question}\label{ques-matr}
Provide an (elegant) algebro-geometric, or linear-algebraic proof of the following
statement: for any three $3 \times 3$ traceless matrices $X$, $Y$, $Z$ in 
general position, the $7$ matrices $XYZ, XZY, YXZ, YZX, ZXY, ZYX$, and 
$(\Tr(XYZ) - \Tr(XZY))E$, are linearly independent.
\end{question}

At the end of this section, let us mention another widely studied problem
amenable to our methods. Namely, people studied a lot Lie-admissible 
third power-associative structures compatible with a given Lie algebra 
structure. (Lie-admissibility alone is too general, and third-power 
associativity is one of the most natural additional conditions leading to a 
meaningful theory). This is known to be equivalent to description of all 
symmetric bilinear maps $\Phi: L \times L \to L$ on a given Lie algebra $L$, 
such that
\begin{equation*}
[\Phi(x,y),z] + [\Phi(z,x),y] + [\Phi(y,z),x] = 0
\end{equation*}
for any $x,y,z\in L$ (see, for example, \cite[\S 1]{jkl}).
The resemblance with the $2$-cocycle equation, as well as with Poissonity
condition, is evident.

\begin{question}
Describe Lie-admissible third power-associative structures on current, extended
current, Kac-Moody Lie algebras, and on Lie algebras of the form 
$\mathsf{sl}_n(A)$, similarly to how it is done for Poisson
structures in this and preceding sections.
\end{question}

This should provide, among other, a conceptual approach to results of 
\cite{jkl}, similarly to those to Kubo's results about Poisson structures 
discussed in \S \ref{sec-poisson-current}. These computations should be pretty 
much straightforward, though, perhaps, technically somewhat challenging at 
places.

\section{Hom-Lie structures}\label{sec-homlie}

Hom-Lie algebras (under different names, notably 
``$q$-deformed Witt or Virasoro'') started to appear long time ago in 
the physical literature, in the constant quest for deformed, in that or another
sense, Lie algebra structures bearing a physical significance. Recently there 
was a surge of interest in them, starting with the paper \cite{hls}.

Recall that a \emph{Hom-Lie algebra} $L$ is an algebra with a skew-symmetric
multiplication $\liebrack$ and a linear map $\varphi:L\to L$ such that the
following generalization of the Jacobi identity, called the 
\emph{Hom-Jacobi identity}, holds:
\begin{equation}\label{hom-lie}
[[x,y],\varphi(z)] + [[z,x],\varphi(y)] + [[y,z],\varphi(x)] = 0
\end{equation}
for any $x,y,z\in L$. 
Further variations of this notion arise by requiring $\varphi$ to be a 
homomorphism, automorphism, etc., of an algebra $L$ with respect to the 
multiplication $\liebrack$.

If $(L,\liebrack)$ is a Lie algebra, then a linear map $\varphi:L\to L$ turning
it into a Hom-Lie algebra -- i.e. such that (\ref{hom-lie}) holds -- is called 
a \textit{Hom-Lie structure on $L$}. The set of all Hom-Lie structures on $L$
will be denoted by $\HomLie(L)$. Obviously, it is a subspace of $\End_K(L)$
containing the identity map.

In this section, the characteristic of the base field $K$ is assumed 
$\ne 2,3$.

\begin{theorem}\label{th-homlie}
Let $L$ be a Lie algebra, $A$ is an associative commutative algebra with unit, 
and one of $L$, $A$ is finite-dimensional. Then 
$$
\HomLie (L\otimes A) \simeq \HomLie(L) \otimes A + 
\set{\varphi\in \End_K(L)}{[[L,L],\varphi(L)] = 0} \otimes \End_K(A) .
$$
Each Hom-Lie structure on $L\otimes A$ can be represented as a sum of 
decomposable Hom-Lie structures of the form $\varphi \otimes \alpha$, 
where $\varphi\in \End_K(L)$, $\alpha\in \End_K(A)$, of the following two
types:
\begin{enumerate}
\item $\varphi \in \HomLie(L)$, $\alpha = \R_u$ for some $u\in A$;
\item $[[L,L],\varphi(L)] = 0$.
\end{enumerate}
\end{theorem}

\begin{proof}
As usual in our approach, decompose a Hom-Lie structure $\Phi$ on 
$L\otimes A$ as $\Phi = \sum_{i\in I} \varphi_i \otimes \alpha_i$ for suitable
linear maps $\varphi_i: L \to L$ and $\alpha_i: A \to A$. The Hom-Jacobi identity
(\ref{hom-lie}) then reads
\begin{equation}\label{yoyo10}
\sum_{i\in I} 
  [[x,y],\varphi_i(z)] \otimes ab\alpha_i(c)
+ [[z,x],\varphi_i(y)] \otimes ca\alpha_i(b) 
+ [[y,z],\varphi_i(x)] \otimes bc\alpha_i(a) = 0
\end{equation}
for any $x,y,z\in L$, $a,b,c\in A$.

Cyclically permuting in this equality $x,y,z$, and summing up the obtained $3$ 
equalities, we get:
$$
\sum_{i\in I} 
\Big([[x,y],\varphi_i(z)] + [[z,x],\varphi_i(y)] + [[y,z],\varphi_i(x)] \Big)
\otimes 
\Big(ab\alpha_i(c) + ca\alpha_i(b) + bc\alpha_i(a)\Big)
= 0 .
$$
Easy transformations show (at this place the assumption that the 
characteristic of the ground field $\ne 2,3$ is essential) that the vanishing 
of the second tensor factor here, 
$ab\alpha_i(c) + ca\alpha_i(b) + bc\alpha_i(a)$, implies the vanishing of 
$\alpha_i$, hence the first tensor factor vanishes for each $\varphi_i$, i.e. each $\varphi_i$
is a Hom-Lie structure on $L$. Writing the latter condition as
$$
[[z,x],\varphi_i(y)] = - [[x,y],\varphi_i(z)] - [[y,z],\varphi_i(x)]
$$
and substituting this back to (\ref{yoyo10}), we get:
$$
\sum_{i\in I} 
  [[x,y],\varphi_i(z)] \otimes \Big(ab\alpha_i(c) - ca\alpha_i(b)\Big)
+ [[y,z],\varphi_i(x)] \otimes \Big(bc\alpha_i(a) - ca\alpha_i(b)\Big)
= 0
$$
for any $x,y,z\in L$, $a,b,c\in A$.

Symmetrizing the last equality with respect to $x,z$, we get:
$$
\sum_{i\in I} 
\Big([[x,y],\varphi_i(z)] - [[y,z],\varphi_i(x)]\Big) \otimes 
\Big(ab\alpha_i(c) - bc\alpha_i(a)\Big)
= 0 .
$$
The vanishing of the first tensor factor here, 
$[[x,y],\varphi_i(z)] - [[y,z],\varphi_i(x)]$, together with the Hom-Jacobi 
identity (\ref{hom-lie}), implies that 
\begin{equation}\label{triv}
[[L,L],\varphi_i(L)] = 0 .
\end{equation}
The vanishing of the second tensor factor, $ab\alpha_i(c) - bc\alpha_i(a)$,
implies that $\alpha_i(a) = a\alpha_i(1)$. Hence we may split the indexing
set into two subsets: $I = I_1 \cup I_2$ such that $\varphi_i$ satisfies
(\ref{triv}) for $i\in I_1$, and $\varphi_i\in \HomLie(L)$ and 
$\alpha_i(a) = au_i$ for some $u_i\in A$, for $i\in I_2$.
It is obvious that for each $i\in I_1$, $i\in I_2$, the decomposable map
$\varphi_i \otimes \alpha_i$ satisfies the Hom-Jacobi identity 
(\ref{yoyo10}), and we are done.
\end{proof}

Hom-Lie structures on simple classical Lie algebras were described in \cite{jl},
and it is natural to try to extend this result to Kac-Moody algebras. However, 
a direct attempt to generalize Theorem \ref{th-homlie} to extended current Lie 
algebras, like it is done for Poisson structures in 
\S \ref{sec-poisson-current}, meets certain technical difficulties.

\begin{question}
Describe Hom-Lie structures on extended current and Kac-Moody Lie algebras,
and on Lie algebras of the form $\mathsf{sl}_n(A)$.
\end{question}

\section{Dual operads and affinizations of Novikov algebras}\label{sec-dual}

The following is a well-known phenomenon from the operadic theory: if $A$ and
$B$ are algebras over binary quadratic operads (Koszul) dual to each other, then 
their tensor product $A \otimes B$ equipped with the bracket 
\begin{equation}\label{eq-ab}
[a \otimes b, a^\prime \otimes b^\prime] = 
aa^\prime \otimes bb^\prime - a^\prime a \otimes b^\prime b
\end{equation}
for $a,a^\prime\in A, b,b^\prime \in B$, becomes a Lie algebra 
(\cite[Theorem 2.2.6(b)]{gk}). The most famous pairs of dual operads are
(Lie, associative commutative) and (associative, associative). In the first 
case, this constructions leads to current Lie algebras, and in the second 
one -- to Lie algebras of the form $(A \otimes B)^{(-)}$ for two associative 
algebras $A$, $B$, of which $\mathsf{gl}_n(A)$ and $\mathsf{sl}_n(A)$ are 
important special cases.

We are going now to describe another broad and interesting class of Lie 
algebras, which fits into this scheme for yet another pair of dual operads.

The ground field $K$ is arbitrary, unless specified otherwise.

Recall that an algebra is called \textit{left Novikov} if it satisfies two identities:
\begin{equation}\label{left-novikov1}
(xy)z - x(yz) = (yx)z - y(xz) 
\end{equation}
and
\begin{equation*}
(xy)z = (xz)y .
\end{equation*}
If an algebra satisfies the opposite identities 
(the order of multiplication is reversed):
\begin{equation}\label{right-novikov1}
(xy)z - x(yz) = (xz)y - x(zy)
\end{equation}
and
$$
x(yz) = y(xz) ,
$$
then it is called \textit{right Novikov}. By just \textit{Novikov algebras} we 
will mean algebras which are either left or right Novikov. 

Every associative commutative algebra is (both left and right) Novikov.
An important feature of Novikov algebras is that they are Lie-admissible.

Let $N$ be a left Novikov algebra, $G$ a commutative semigroup (written multiplicatively),
and $\chi: G \to K$ a map.
Define a bracket on the tensor product $N\otimes K[G]$, where $K[G]$ is the semigroup
algebra, as follows:
\begin{equation}\label{eq-brack}
[x \otimes a, y \otimes b] = \Big(\chi(a)xy - \chi(b)yx\Big) \otimes ab ,
\end{equation}
where $x,y \in N$ and $a,b \in G$.
This bracket is obviously anticommutative.

\begin{lemma}\label{lemma-jacobi}
For a fixed $G$ and $\chi$ as above, the bracket {\rm (\ref{eq-brack})} 
satisfies the Jacobi identity for an arbitrary left Novikov algebra $N$ if an 
only if
\begin{equation}\label{eq-char}
\chi(ab) - \chi(ac) = \chi(b) - \chi(c)
\end{equation}
for any $a,b,c\in G$.
\end{lemma}

\begin{proof}
A tedious, but straightforward check.
\end{proof}

\begin{remark-nonum}
If $G$ contains a unit $e$, then the property (\ref{eq-char}) is equivalent to
\begin{equation*}
\chi(ab) = \chi(a) + \chi(b) - \chi(e) 
\end{equation*}
for any $a,b\in G$.
\end{remark-nonum}

A map satisfying the property (\ref{eq-char}) will be called a 
\textit{quasi-character of $G$}. If $\chi$ is a quasi-character of a commutative
semigroup $G$, then the Lie algebra defined by the bracket (\ref{eq-brack})
will be denoted as $N_\chi[G]$.
Let us list a few interesting special cases of this construction.

\smallskip

(i) \textit{Witt algebras}.
$G = (\mathbb Z, +)$ (so $K[G] \simeq K[t,t^{-1}]$, the algebra of Laurent 
polynomials), $\chi =$ degree of the monomial, 
and $N = K$.
This is the famous infinite-dimensional (two-sided) Witt algebra $W$ with the basis 
$\set{e_m}{m\in \mathbb Z}$ and multiplication 
$[e_m,e_n] = (m-n)e_{m+n}$, $m,n\in \mathbb Z$. A variation of this 
constructions assumes $K$ has characteristic $p>0$, and 
$G = \mathbb Z/p\mathbb Z$ (so $K[G] \simeq K[t]/(t^p)$), leading to 
the $p$-dimensional Witt algebra.

\smallskip

(ii) \textit{Current algebra over a Witt algebra}.
More generally, taking in the previous example $N$ as an arbitrary associative 
commutative algebra instead of $K$, we get the tensor product of the corresponding Witt 
algebra with $N$.

\smallskip

(iii) \textit{Skew-symmetrization of a Novikov algebra}.
$G = \{1\}$ and $\chi(1) = 1$. This is the Lie algebra $N^{(-)}$, a
skew-symmetrization of $N$.

\smallskip

(iv) \textit{Current algebra over a skew-symmetrization of a Novikov algebra}.
More generally, letting $G$ be arbitrary and $\chi$ is identically $1$,
we get the current Lie algebra 
$$
N^{(-)} \otimes K[G] \simeq (N \otimes K[G])^{(-)} .
$$

\smallskip

(v) \textit{Affinization of a Novikov algebra}.
$G = (\mathbb Z, *)$, where $n*m = n+m-1$ (so $K[G] \simeq K[t,t^{-1}]$ with 
multiplication $t^m t^n = t^{n+m-1}$), $\chi =$ degree of the monomial. The 
resulting Lie bracket is
\begin{equation*}
[x \otimes t^m, y\otimes t^n] = (mxy - nyx) \otimes t^{m+n-1} ,
\end{equation*}
where $x,y\in N$, $n,m\in \mathbb Z$.
This construction, which appeared (in a somewhat implicit form) in the 
pioneering papers \cite{gelf-dorfman} and \cite{novikov} (in the latter, under
the name ``Poisson brackets of hydrodynamic type''), appears also in various 
physical questions, and in the theory of vertex operator algebras (see 
references in \cite{pei-bai2010}, \cite{pei-bai}). It resembles the 
construction of untwisted affine Kac-Moody algebras, and we will call it 
\textit{affinization of the Novikov algebra $N$}.

\smallskip

(vi) \textit{The Heisenberg--Virasoro algebra} (see \cite{pei-bai2010}).
Specializing example (v) to the case where $N$ is the algebra with the 
basis $\{ x,y \}$ and multiplication table
$$\begin{array}{c|cc}
   & x & y          \\
\hline
x  & x & 0    \\
y  & y & 0   
\end{array}$$
we get a nilpotent extension of the two-sided Witt algebra
$\langle e_m, h_m \,|\, m\in \mathbb Z\rangle$ with multiplication table
\begin{equation}\label{eq-heisen}
[e_m,e_n] = (m-n)e_{m+n} , \qquad [e_m,h_n] = -n h_{m+n} , \qquad [h_m,h_n] = 0
\end{equation}
for $m,n\in \mathbb Z$.

\smallskip

(vii) \textit{The Schr\"odinger--Virasoro algebra} (see \cite{pei-bai}).
Specializing example (v) to the case where $N$ is the algebra with the basis 
$\{ x,y,z \}$ and multiplication table
$$\begin{array}{c|ccc}
   & x & y & z          \\
\hline
x  & x & 0 & \frac 12 z \\
y  & y & 0 & 0          \\
z  & z & y & 0
\end{array}$$
we get a nilpotent extension of the two-sided Witt algebra
$\langle e_m, h_m, f_r \,|\, m\in \mathbb Z, r \in \frac 12 + \mathbb Z \rangle$
with multiplication rules between the basic elements, in addition to relations 
(\ref{eq-heisen}), as follows:
\begin{equation*}
[e_m,f_r] = \big(\frac{m}{2} - r\big)f_{m+r} , \qquad 
[f_r,f_s] = (r-s)h_{r+s} , \qquad 
[h_m,f_r] = 0 
\end{equation*}
for $m\in \mathbb Z$, $r,s\in \frac 12 + \mathbb Z$.

\smallskip

As we see, most of the interesting instances of $N_\chi[G]$ are 
infinite-dimensional, physically-motivated Lie algebras of characteristic zero.
We suggest that this construction may serve as an organizing principle in an
entirely different, at the first glance, context -- in the unruly world of 
simple finite-dimensional Lie algebras over fields of small characteristics.

\begin{question}
Is it possible to realize the so-called bi-Zassenhaus algebras over a field of 
characteristic $2$ introduced in \cite{jurman}, as Lie algebras of the form
$N_\chi[G]$ for suitable $N$, $G$ and $\chi$? 
\end{question}

Bi-Zassenhaus algebras form a family parametrized by two integers 
$g\ge 2, h\ge 1$ and are constructed as follows. Consider a Lie algebra with 
the basis 
$\set{e_{(i,\alpha)}}{i\in \mathbb Z/2\mathbb Z, \alpha\in \mathbb Z/2^{g+h}\mathbb Z}$
and multiplication table:
\begin{align*}
[e_{(0,\alpha)}, e_{(0,\beta)}] &= (\alpha + \beta) e_{(0,\alpha + \beta)} \\
[e_{(0,\alpha)}, e_{(1,\beta)}] &= (\alpha + \beta) e_{(1,\alpha + \beta)} \\
[e_{(1,\alpha)}, e_{(1,\beta)}] &= (\alpha^{2^g-1} + \beta^{2^g-1}) e_{(0,\alpha + \beta)} .
\end{align*}
This algebra possess an ideal which is a simple Lie algebra looking, at the 
first glance, very different from the simple modular Lie algebras of Cartan 
type\footnote{
However, recently is was shown that it is isomorphic to a suitable Hamiltonian 
algebra under a highly non-trivial isomorphism, see \cite{grishkov} and
\cite[\S 3]{leites-jurman}.
}. Note that it is not a skew-symmetrization of a Novikov algebra, so $G$ in 
such realization should contain more than one element. On the other hand, 
dimension of $N$ should be at least $3$, as a seemingly direct approach -- 
take a $2$-dimensional Novikov algebra to parametrize the first coordinate in 
the set of indices -- does not work, as a quick glance at the list of such 
algebras (for example, in \cite[\S 2]{bai-meng}) reveals.

\begin{question}
The same question for Melikyan algebras -- a series of simple Lie algebras over
a field of characteristic $5$ which are neither of classical, nor of Cartan 
type, and is peculiar to that characteristic (see \cite[\S 4.3]{strade}).
\end{question}

This question is, apparently, more difficult.

\bigskip

It is remarkable (though pretty much straightforward) that, actually, the bracket (\ref{eq-brack}) is a particular 
case of (\ref{eq-ab}). Namely, as noted in \cite[\S 4]{dzhu-novikov}, 
the left Novikov and right Novikov operads are dual to each other, hence if $A$
is a left Novikov algebra and $B$ is a right Novikov one, their tensor product
$A \otimes B$ equipped with the bracket (\ref{eq-ab}), is a Lie algebra.
If $G$ is a commutative semigroup and $\chi$ its quasi-character, then the semigroup algebra 
$K[G]$ equipped with multiplication 
$a \cdot b = \chi(a)ab$ for $a,b\in G$, becomes a right Novikov algebra,
and the bracket (\ref{eq-brack}) becomes a particular case of the bracket 
(\ref{eq-ab}) with $A = N$ and $B = (K[G], \cdot)$.

\bigskip

In \cite{pei-bai2010}, \cite{pei-bai} it was demonstrated how central extensions
of algebras from the examples (vi) and (vii) above, can be realized in terms
of some bilinear forms on the underlying Novikov algebra $N$. We suggest that
these results are instances of a more general principle, which is a 
generalization of results of \cite{low}, \cite{without-unit} and 
\S\S {\rm \ref{sec-poisson-current}-\ref{sec-homlie}} of this paper, from 
current Lie algebras and Lie algebras of the form $\mathsf{sl}_n(A)$, to Lie 
algebras with the bracket (\ref{eq-ab}):

\begin{question}\label{quest-dual}
Describe, as much as possible, low-degree (co)homology, invariant bilinear 
forms, Poisson structures, Hom-Lie structures on Lie algebras of the form
$(A \otimes B)^{(-)}$, where $A$, $B$ are algebras over dual binary
quadratic operads, in terms of the underlying algebras $A$ and $B$.
\end{question}

As we have seen at the beginning of \S \ref{sec-poisson-gl}, even in the 
``classical'' case, when both $A$ and $B$ are associative, this can be 
unfeasible in some cases, so it is interesting to understand which concrete
properties of operads turn that or another problem of that type to a tractable one.

At least, we expect that all this is entirely tractable for algebras of the form
$N_\chi[G]$:

\begin{question}\label{quest-nchi}
Describe low-degree (co)homology, invariant bilinear forms, Poisson structures,
Hom-Lie structures on Lie algebras of the form $N_\chi[G]$ in terms of 
$N$ and $G$.
\end{question}

Specializing these conjectural description further to the case (v) above,
we should get the corresponding results for affinizations of Novikov algebras, 
similarly to how the results for Kac-Moody algebras are derived from those for
current Lie algebras.

\bigskip

We conclude with yet another couple of problems.

Arguing from the physical perspective, Bai, Meng and He in \cite{fermionic} called 
the left Novikov algebras \textit{bosonic}, and introduced a new class of so-called
fermionic Novikov algebras as algebras satisfying two identities:
the left-symmetric identity (\ref{left-novikov1}), and
$$
(xy)z = - (xz)y .
$$
We will call such algebras \textit{left fermionic Novikov}. The dual operad to the
left fermionic Novikov operad is \textit{right fermionic Novikov}, governed by two
identities: the right-symmetric identity (\ref{right-novikov1}), and
$$
x(yz) = - y(xz) .
$$

\begin{question}
Are there ``interesting'' Lie algebras realized as {\rm (\ref{eq-ab})} in this 
context, i.e. as skew-sym\-me\-tri\-za\-ti\-on of the tensor product of left 
and right fermionic Novikov algebras?
\end{question}

\begin{question}
To superize constructions of this section. Consider Problems \ref{quest-dual} 
and \ref{quest-nchi} for superalgebras.
\end{question}

We expect that this conjectural superization will provide an alternative view
(and probably much more) on the so-called stringy Lie superalgebras, see 
\cite{kac-leur}, \cite{vietnamica} and \cite[Appendix D3, \S 8]{berezin}.

\section*{Acknowledgements}

Thanks are due to 
Rutwig Campoamor-Stursberg, 
Askar Dzhumadil'daev,
Ale\-xan\-der Grishkov,
Dimitry Leites,
Johan van de Leur,
Nikolai Vavilov and Yuri Zarhin
for useful discussions and/or bringing important papers to my attention.
Thanks are also due to the anonymous referee for pointing out a few 
inaccuracies in the text and suggesting many improvements.

Sage software (\texttt{http://sagemath.org/}) version 4.3.2 was utilized
to make conjectures and to verify some calculations performed in this paper.

This work was supported by grants ERMOS7 (Estonian Science Foundation
and Marie Curie Actions) and ETF9038 (Estonian Science Foundation).
Part of this work was done during my visit to IH\'ES at February 2012.

\end{document}